\documentclass[12pt]{amsart}
\usepackage{amsmath}
\usepackage{amssymb}
\usepackage{amsthm}
\usepackage{amscd}
\usepackage[all]{xy}
\usepackage{color}

\usepackage{graphicx}

\usepackage[textwidth=6in, textheight=8.5in,marginpar=0.75in]{geometry}

\newtheorem{thm}{Theorem}[section]

\newtheorem{lem}[thm]{Lemma}

\theoremstyle{remark}

\theoremstyle{definition}
\newtheorem{defn}[thm]{Definition}
\newtheorem{rmk}[thm]{Remark}

\newcommand{\emptyword}{\lambda}
\newcommand{\groupid}{1}
\newcommand{\equalword}{=}
\newcommand{\equalinf}{=_F}
\newcommand{\Nf}{\ensuremath{\mathcal{N}}}
\newcommand{\Nfgs}{{\mathcal{N}}}  % subscript by gs?
\def\nf#1{\mathsf{nf}(#1)}
\newcommand{\lbl}{{\mathsf{label}}}  
\newcommand {\lab} {{\mathsf{label}}}
\newcommand{\Graph}[1]{\mathsf{graph}\left({#1}\right)}
\newcommand{\ra}{\rightarrow}

\begin{document}

\title{Autostackability of Thompson's group~$F$}

\author[N.~Corwin]{Nathan Corwin}
\address{Department of Mathematics and Statistics, Villanova University, 
800 Lancaster Avenue, Villanova, PA 19085, USA}
\email{nathan.corwin@villanova.edu}

\author[G.~Golan]{Gili Golan}
\address{Department of Mathematics, Ben-Gurion University of the Negev, P.O.B.~653, Be$'$er Sheva 84105, Israel}
\email{golangi@bgu.ac.il}

\author[S.~Hermiller]{Susan Hermiller}
\address{Department of Mathematics, University of Nebraska, Lincoln, NE 68588-0130, USA}
\email{hermiller@unl.edu}

\author[A.~Johnson]{Ashley Johnson}
\address{Department of Mathematics, University of North Alabama, Florence, AL 35632, USA}
\email{ajohnson18@una.edu}

\author[Z.~\v{S}uni\'c]{Zoran \v{S}uni\'c}
\address{Department of Mathematics, Hofstra University, Hempstead, NY 11549, USA}
\email{zoran.sunik@hofstra.edu}


%%%%%%%%%%%%%%%%%%%%%%%%%%%%%%%%%%%%%%%%%%%%%%%%%%%%%%%%%%%%%%%%%%

\begin{abstract}
The word problem for Thompson's group $F$ has a solution, but 
it remains unknown whether $F$ is automatic or has a finite or regular convergent
(terminating and confluent) rewriting system.  We show that the group $F$
admits a natural extension of these two properties, namely autostackability,
and we give an explicit bounded regular convergent
prefix-rewriting system for $F$.
\end{abstract}

\maketitle

%%%%%%%%%%%%%%%%%%%%%%%%%%%%%%%%%%%%%%%%%%%%%%%%%%%%%%%%%%%%%%%%%%%%%%%%
%%%%%%%%%%%%%%%%%%%%%%%%%%%%%%%%%%%%%%%%%%%%%%%%%%%%%%%%%%%%%%%%%%%%%%%%

\section{Introduction}\label{sec:intro}

%%%%%%%%%%%%%%%%%%%%%%%%%%%%%%%%%%%%%%%%%%%%%%%%%%%%%%%%%%%%%%%%%%%%%%%%
%%%%%%%%%%%%%%%%%%%%%%%%%%%%%%%%%%%%%%%%%%%%%%%%%%%%%%%%%%%%%%%%%%%%%%%%

The group of orientation-preserving piecewise linear
automorphisms of the unit interval for which all linear
slopes are powers of 2, and all breakpoints lie in
the 2-adic numbers, is known as {\em Thompson's group $F$}. The group
$F$ has a well-known finite presentation
\[
F=\langle x,y \mid 
[y,xyx^{-2}], [y,x^{2}yx^{-3}] \rangle
\] 
%\[
%F=\langle x_0,x_1 \mid 
%[x_0x_1^{-1},x_0^{-1}x_1x_0], [x_0x_1^{-1},x_0^{-2}x_1x_0^2] \rangle
%\]
(and the generators $x$ and $y$ 
are the elements $x_0$ and $x_1$, respectively, of
a standard infinite generating set).
The group $F$
has been the focus of considerable
research in recent years, because of its
connections to many other fields, 
because of long-standing open problems
such as the amenability of $F$, and
because of its many surprising properties
(for example, $F$ is an infinite dimensional
torsion-free group with homological type $FP_\infty$~\cite{browng}).

%Thompson's group $F$ has been the focus of considerable
%research in recent years, and

Many algorithmic problems have been shown to have
solutions for $F$.
The word problem is solvable for $F$;
moreover, Guba~\cite{guba}
shows that the Dehn function for $F$ is quadratic.
Guba and Sapir~\cite{gsdiag} 
prove that the conjugacy problem is solvable.
Kassabov and Matucci~\cite{km} show that
the simultaneous conjugacy problem is solvable,
and Burillo, Matucci, and Ventura~\cite{bmv}
show that the twisted conjugacy problem also has a solution.
Golan and Sapir~\cite{GoSa} give a solution to the
subgroup membership problem for a large family of subgroups of $F$.
In~\cite{bbh}, Bleak, Brough and Hermiller show that there is an
algorithm that, upon input of a finite set $S$ of elements
of $F$ (that is, words over the generating set 
$A:=\{x^{\pm 1},y^{\pm 1}\}$), can determine whether
or not the subgroup $\langle S \rangle$ generated by $S$ is 
solvable and, if it is, also determines its derived length.
In~\cite{golan}, Golan shows that there is an algorithm that, upon
input of the finite subset $S \subset F$, can determine whether
$S$ generates the entire group $F$.

On the other hand, many algorithmic questions also
remain open for Thompson's group $F$,
including  the questions of
whether $F$ is an automatic group (see~\cite{echlpt} for definitions
and background on automatic groups), and whether $F$ admits
a finite or regular convergent rewriting system 
(defined in~Section~\ref{subsec:regular});
see the problem list~\cite{thomp}.
Both of these are questions about the complexity of the
word problem for $F$; in particular, about whether the
word problem can be solved using a computer with a finite
amount of memory, known as a finite state automaton.

Both (prefix-closed) automatic structures and finite 
convergent rewriting systems are special cases of
convergent prefix-rewriting systems.  A {\em convergent
prefix-rewriting system}, or {\em CP-RS}, for a group $G$
consists of a finite set $A$ and a set 
$R \subset A^* \times A^*$ of ordered pairs of words
over $A$ such that $G$ is presented, as a monoid, by
\[
G = Mon \langle A \mid \{u=v \mid (u,v) \in R\} \rangle,
\]
and the rewritings $uz \ra vz$ for all $(u,v) \in R$ and $z \in A^*$
satisfy
\begin{itemize}
\item (termination) there is no infinite sequence of rewritings
$x_1 \ra x_2 \ra \cdots$, and  
\item (confluence) the set $Irr(R)$ 
of words that cannot be rewritten
is a set of (unique) normal forms for $G$.
\end{itemize}
A CP-RS is {\em regular} if the subset $R \subset A^* \times A^*$ 
is (synchronously) regular; that is, if the set of padded
words obtained from $R$ can be recognized by
a finite state automaton (see Section~\ref{subsec:regular} for
more information on regular sets, finite state automata, and padded words).  
Otto~\cite{otto} has shown that a group $G$ has a 
prefix-closed automatic structure if and only if
$G$ admits a regular CP-RS  $R$ satisfying the further
condition that whenever $(u,v) \in R$ then $v \in Irr(R)$
and $u \in Irr(R \setminus \{(u,v)\})$.
The group $G$ has a finite convergent rewriting system 
if and only if there is a finite set
$R' \subset A^* \times A^*$ such that
$\{(wu',wv') \mid w \in A^*,~(u',v') \in R'\}$ is a
CP-RS for $G$.  

In~\cite{BHH:algorithms}, Brittenham, Hermiller, and Holt
considered another solution to the word problem for groups
by finite state automata that is a natural extension of 
both of automaticity and finite rewriting systems, 
namely autostackability.  
A CP-RS $R$ is {\em bounded} if
there is a constant $K>0$ such that whenever $(u,v) \in R$, then
$(u,v)=(wu',wv')$ for some $w,u',v' \in A^*$ with $l(u')+l(v') \le K$.
A group $G$ is {\em autostackable} if $G$ has a bounded
regular convergent prefix-rewriting system.
Every group that has a prefix-closed (asynchronous or synchronous)
automatic structure, or has a finite convergent rewriting system,
is autostackable~\cite{BHH:algorithms}.

In analogy with the characterization of automatic groups by
a regular set of normal forms with a $K$-fellow traveler property,
autostackability of a group $G$ with a finite
inverse-closed generating set $A$
also has a topological characterization,
in terms of a discrete dynamical system on the Cayley graph
$\Gamma:=\Gamma_A(G)$ of $G$ over $A$, as follows.
A \emph{flow function} for $G$ with \emph{bound} $K \ge 0$,
with respect to
a spanning tree $T$ in $\Gamma$,
%a constant $K \ge 0$,
%and a \emph{flow function}
is a function $\Phi$ mapping the
set $\vec E$ of directed edges of $\Gamma$ to the
set $\vec P$ of directed paths in $\Gamma$, such that
\begin{itemize}
        \item (fixing the tree) for each $e \in \vec E$ the path
	$\Phi(e)$ has the same initial and terminal
	vertices as $e$ and length at most $K$, and
	$\Phi$ acts as the identity on edges lying
	in $T$, and
	\item (termination)  there is no infinite sequence
	$e_1,e_2,e_3,...$ of edges with each $e_i$
	not in $T$ and each $e_{i+1}$ in the path $\Phi(e_i)$.
\end{itemize}
Extending $\Phi$ to $\widehat \Phi:\vec P \to \vec P$ by $\widehat
\Phi(e_1 \cdots e_n):=\Phi(e_1) \cdots \Phi(e_n)$, 
%where $\cdot$ denotes concatenation of paths, 
then for all $p \in \vec P$, there
is a $n_p \in \mathbb N$ such that $\widehat \Phi^{n_p}(p)$ is a path in
the tree $T$; that is, when $\widehat \Phi$ is iterated, paths in
$\Gamma$ ``flow'' toward the tree. 
%A finitely generated group admitting
%a bounded flow function over some finite set of generators is
%called \emph{stackable}.
Let $\Nf_T$ denote the set of words
labeling non-backtracking paths in $T$ starting at the vertex
labeled by the identity %$\groupid$ 
of $G$,
% (hence $\Nf_T$ is a
%prefix-closed set of normal forms for $G$), 
and let $\lbl:\vec P \to
A^*$ be the function that returns the label of any directed path in
$\Gamma$. 
The flow function $\Phi$ is \emph{regular} if the
graph of $\Phi$, written as a subset of $A^* \times A^* \times A^*$ by
\begin{align*}
\Graph{\Phi}:=\{(\gamma,a,\lbl(\Phi(e_{\gamma,a}))) \mid &
\gamma \in \Nf_T, a \in A, \text{ and }
e_{\gamma,a} \in \vec E
\text{ has } \\
& \text{ initial vertex } \gamma \text{ and label }a\},
\end{align*}
is regular.
%recognized by a finite state automaton
%(that is, $\Graph{\Phi}$ is a
%regular language).
%We say that a finitely generated group $G$ is
%(auto)stackable, without reference to a generating set, if there
%exists a finite generating set $A$ for $G$ such that $G$ is
%(auto)stackable in the generators $A$.
A group $G$ is autostackable if and only if $G$ admits
a regular bounded flow function~\cite{BHH:algorithms}.

Several convergent prefix-rewriting systems 
for Thompson's group $F$, over the generating set 
$A=\{x^{\pm 1},y^{\pm 1}\}$, have been described in earlier
papers; in each case the CP-RS is 
computable (can be recognized by a Turing machine), but is
not regular.
In~\cite{chst}, Cleary, Hermiller, Stein and Taback
show that Thompson's group $F$ has a bounded CP-RS $R$
(by constructing a bounded flow function)
for which the set $R$ is computable,
and the normal form set $Irr(R)$ is a set of words
labeling quasigeodesics in the Cayley graph that is
not regular, but that is a context-free language (the next
simplest language class in the Chomsky hierarchy of 
computable languages).
Guba and Sapir~\cite{GuSa} give a convergent
rewriting system $R$ for $F$ for which the set $R$ is
computable but not context-free (and hence not regular),
but the normal form set $Irr(R)$ is the regular set
\[
\Nfgs = A^* \setminus A^*\left(\{aa^{-1} \mid a \in A\} \cup
\{y^\epsilon xx^*y,y^\epsilon x^2x^*y^{-1} \mid \epsilon\in\{1,-1\}\}
  \right)A^*;
\]
that is, $\Nfgs$ is the set of all freely reduced words over $A$
that do not contain a subword of the form 
$y^\epsilon x^iy$ or $y^\epsilon x^{i+1}y^{-1}$ for any
$\epsilon\in\{1,-1\}$ and $i \geq 1$.
Section~\ref{subsec:gsrs} contains more details on the
Guba-Sapir system.

The main result of this paper is the following.

\begin{thm}\label{thm:fisautostk}
Thompson's group~$F$ is autostackable.  Moreover, the following
is a bounded regular convergent prefix-rewriting system for $F$:
\begin{eqnarray*}
R &=&  \cup_{a \in A} \left\{uaa^{-1} \ra u \mid u \in A^*\right\}  \\
& & \cup \left(\cup_{\epsilon\in\{1,-1\}} \cup_{i \in \{1,2\}}
\left\{uy^\epsilon x^iy \ra 
    ux^{i}yx^{-i-1}y^{\epsilon}x^{i+1} 
    \mid uy^\epsilon x^i \in \Nfgs\right\} \right) \\
& & \cup \left(\cup_{\epsilon\in\{1,-1\}} \cup_{i \in \{2,3\}}
\left\{uy^\epsilon x^iy^{-1} \ra  
    ux^{i}y^{-1}x^{-i+1}y^{\epsilon}x^{i-1} \mid
  uy^\epsilon x^i \in \Nfgs\right\} \right) \\
& & \cup \left\{uxy \ra   uy^{-1}xyx^{-2}yx^2 \mid
u \in \Nfgs \cap (A^*y^{\pm 1}A^*x^2)\right\}  \\
& & \cup \left\{ux^2y^{-1} \ra  u y^{-1}x^2y^{-1}x^{-1}yx \mid
u \in \Nfgs \cap (A^*y^{\pm 1}A^*x^2)\right\}.
\end{eqnarray*}
\end{thm}

%%
%& & \{uy^\epsilon x^iy \ra 
%    ux^{i}yx^{-i-1}y^{\epsilon}x^{i+1} \mid
%\epsilon\in\{1,-1\},~1 \le i \le 2, 
%           \text{ and } uy^\epsilon \in \Nfgs\} \bigcup \\
%%
%& & \{uy^\epsilon x^iy^{-1} \ra  
%    \mid ux^{i}y^{-1}x^{-i+1}y^{\epsilon}x^{i-1} \mid
%\epsilon\in\{1,-1\},~2 \le i \le 3, 
%            \text{ and } uy^\epsilon \in \Nfgs\} \bigcup \\

%%
%& & \{ux^3y \ra 
%    ux^{2}y^{-1}xyx^{-2}yx^2 \mid
%\epsilon\in\{1,-1\},~3 \le i, \
%          \text{ and } u \in \Nfgs \cap (A^*y^{\pm 1}A^*)\} \cup \\
%%
%& & \{ux^4y^{-1} \ra 
%    u x^{2}y^{-1}x^2y^{-1}x^{-1}yx \mid
%\epsilon\in\{1,-1\},~4 \le i, 
%           \text{ and } u \in \Nfgs \cap (A^*y^{\pm 1}A^*)\}.

%%
%& & \{uy^\epsilon x^iy \ra 
%    uy^\epsilon x^{i-1}y^{-1}xyx^{-2}yx^2 \mid
%u \in A^*, \epsilon\in\{1,-1\},~3 \le i, \
%          \text{ and } uy^\epsilon \in \Nfgs\} \cup \\
%%
%& & \{uy^\epsilon x^iy^{-1} \ra 
%    uy^\epsilon x^{i-2}y^{-1}x^2y^{-1}x^{-1}yx \mid
%u \in A^*, \epsilon\in\{1,-1\},~4 \le i, 
%           \text{ and } uy^\epsilon \in \Nfgs\}.

%x^{-1}y^{-1}xyx^{-2}yx^2 & \text{ if } b=1 \text{ and } i > 2 \\
%x^{-i}y^{-\epsilon}x^{i}yx^{-i-1}y^{\epsilon}x^{i+1} & 
%                   \text{ if } b=1 \text{ and } 1 \le i \le 2 \\
%x^{-2}y^{-1}x^2y^{-1}x^{-1}yx & \text{ if } b=-1 \text{ and } i > 3 \\
%x^{-i}y^{-\epsilon}x^{i}y^{-1}x^{-i+1}y^{\epsilon}x^{i-1} & 
%               \text{ if } b=-1 \text{ and } 2 \le i \le 3 \\

In essence, this system contains only finitely many rewriting rules
$u' \ra v'$, but when a rule can be applied is determined by the prefix, 
in normal form, preceding the subword on which the rule is to be applied.
More particularly, for each of the finitely many rewritings $u' \ra v'$, 
there is a regular language $L_{(u',v')}$
such that the rewriting $wu'z \ra wv'z$ can only be applied if $w \in L_{(u',v')}$.
Thus the word problem for Thompson's group $F$ 
can be solved by a finite state machine. 

We note that the set of
normal forms for the prefix-rewriting system in Theorem~\ref{thm:fisautostk}
is the same as the set of normal forms of the Guba-Sapir rewriting system;
however, in contrast to the Guba-Sapir rules,
the prefix-rewriting system in Theorem~\ref{thm:fisautostk} is regular,
and at each step a word of length at most 5 is replaced by
a word of length at most 10.
In our proof of Theorem~\ref{thm:fisautostk}, we consider the
topological view of autostackability, and show that
Thompson's group $F$ has a regular bounded flow function.

The paper is organized as follows.
We begin in Section~\ref{sec:background} with notation used throughout
the paper, background on regular
languages (in Section~\ref{subsec:regular}), and a more
detailed discussion of the Guba-Sapir rewriting system 
and its normal form set (in Section~\ref{subsec:gsrs}).
Section~\ref{sec:prelim} contains an analysis of the relationship
between the normal forms (in the set $\Nfgs$ above) for the two endpoints 
of an edge of the Cayley graph $\Gamma_A(F)$ of $F$.
In Section~\ref{sec:order}, we define a well-founded strict
partial order on the edges of $\Gamma$, which is used (in Section~\ref{sec:proof})
in the proof of termination for the flow function for $F$.
In Section~\ref{sec:interlude}, we give a pictorial motivation
for the choice of this strict partial order.
Finally, Section~\ref{sec:proof} contains the proof of Theorem~\ref{thm:fisautostk}.

%%%%%%%%%%%%%%%%%%%%%%%%%%%%%%%%%%%%%%%%%%%%%%%%%%%%%%%%%%%%%%%%%%%%%%%%%%%%%
%%%%%%%%%%%%%%%%%%%%%%%%%%%%%%%%%%%%%%%%%%%%%%%%%%%%%%%%%%%%%%%%%%%%%%%%%%%%%

\section{Background and notation}\label{sec:background}

%%%%%%%%%%%%%%%%%%%%%%%%%%%%%%%%%%%%%%%%%%%%%%%%%%%%%%%%%%%%%%%%%%%%%%%%%%%%%
%%%%%%%%%%%%%%%%%%%%%%%%%%%%%%%%%%%%%%%%%%%%%%%%%%%%%%%%%%%%%%%%%%%%%%%%%%%%%

Let $x:=x_0$ and $y:=x_1$ be the generators of Thompson's group~$F$,
and let $A:=\{x^{\pm 1},y^{\pm 1}\}$.
Let $A^*$ be the set of all words over $A$.
%; that is,
%$A^*$ is the set of all finite strings over $A$, including the empty word.
Let $\groupid$ denote the identity element of $F$, and let $\emptyword$
denote the empty word in $A^*$.  For two words $v,w \in A^*$, we 
write $v \equalword w$ if $v$ and $w$ are the same word in $A^*$, and
$v \equalinf w$ if $v$ and $w$ represent the same element of $F$.

Let $\Gamma=\Gamma_A(F)$ be the Cayley graph of Thompson's group $F$
over the generating set $A$.  %=\{x^{\pm 1},y^{\pm 1}\}$.
Given any directed edge $e$ of $\Gamma$, let $e_-$ and $e_+$ denote the elements
of $F$ at the initial
and terminal vertices of $e$, respectively,
and let $\lab(e)$ denote the label of the edge $e$.
We denote $e$ by $e_{\gamma,a}$ 
where $a = \lab(e)$ and $\gamma \in A^*$ is any word satisfying 
$\gamma \equalinf e_-$.

%%%%%%%%%%%%%%%%%%%%%%%%%%%%%%%%%%%%%%%%%%%%%%%%%%%%%%%%%%%%%%%%%%%%%%%%
%%%%%%%%%%%%%%%%%%%%%%%%%%%%%%%%%%%%%%%%%%%%%%%%%%%%%%%%%%%%%%%%%%%%%%%%

\subsection{Synchronously regular languages and rewriting systems}\label{subsec:regular}

%%%%%%%%%%%%%%%%%%%%%%%%%%%%%%%%%%%%%%%%%%%%%%%%%%%%%%%%%%%%%%%%%%%%%%%%
%%%%%%%%%%%%%%%%%%%%%%%%%%%%%%%%%%%%%%%%%%%%%%%%%%%%%%%%%%%%%%%%%%%%%%%%

Details and proofs of the contents of this section
can be found in~\cite{echlpt,hu,BHH:algorithms}.

Let $A$ be a finite set.
The set of all finite strings over $A$
(including the empty word $\emptyword$)
is written $A^*$.
A \emph{language} is a subset
$L\subseteq A^*$. Given
languages $L_1,L_2$
the \emph{concatenation} $L_1L_2$
of $L_1$ and $L_2$ is
the set of all expressions of the form $l_1l_2$ with $l_i\in L_i$.
The \emph{Kleene star} of $L$, denoted $L^*$,
is the union of $L^n$ over all integers $n\ge 0$.

The class of \emph{regular languages}
over $A$ is the smallest class of languages
that contains all finite languages and is closed under union,
intersection, concatenation, complement and Kleene star. (Note that
closure under some of these operations is redundant.)
Regular languages are precisely those accepted by finite state
automata; that is, by computers with a bounded amount of memory.
%More precisely, a finite state automaton consists of a finite set of
%states $Q$, an initial state $q_0 \in Q$, a set of accept states $P
%\subseteq Q$, a finite set of letters $A$, and a transition function
%$\delta:Q \times A \ra Q$.  The map $\delta$ extends to a function
%$\delta:Q \times A^* \ra Q$; for a word $w=a_1 \cdots a_k$ with each
%$a_i$ in $A$, the transition function gives
%$\delta(q,w)=\delta(\cdots (\delta(\delta(q,a_1),a_2),\cdots,a_k)$.
%The automaton can also considered as a directed labeled graph whose
%vertices correspond to the the state set $Q$, with a
%directed edge from $q$ to $\delta(q,a)$ labeled by $a$
%for each $a \in A$ and $q \in Q$. Using this
%model $\delta(q,w)$ is the terminal vertex of the path starting at
%$q$ labeled by $w$. A word $w$ is in the language of this automaton
%if and only if $\delta(q_0,w) \in P$.

The concept of regularity is extended to
subsets of a Cartesian product
$(A^*)^n=A^* \times \cdots \times A^*$ of $n$ copies
of $A^*$ as follows.
Let $\$$ be a symbol not contained in $A$.
Given any tuple $w=(a_{1,1} \cdots a_{1,m_1},...,
a_{n,1} \cdots a_{n,m_n}) \in (A^*)^n$ (with
each $a_{i,j} \in A$), rewrite
$w$ to a \emph{padded word} $\hat w$ over the finite
alphabet $B:=(A \cup \$)^n$ by
$\hat w :=
(\hat a_{1,1},...,\hat a_{n,1}) \cdots (\hat a_{1,N},...,\hat a_{n,N})$
where $N=\max\{m_i\}$ and $\hat a_{i,j} =a_{i,j}$
for all $1 \le i \le n$
and $1 \le j \le m_i$ and $\hat a_{i,j}=\$$ otherwise.
A subset $L \subseteq (A^*)^n$ is called
a \emph{regular language} (or, more precisely,
\emph{synchronously regular}) if the set
$\{ \hat w \mid w \in L\}$ is a regular subset
of  $B^*$.

The following lemma, much of the proof of which can be found
in \cite[Chapter 1]{echlpt}, contains  closure
properties of regular languages that are used below in Section~\ref{sec:proof}.

\begin{lem}\label{lem:regclosure}
Let $A$ be a finite set, let $z$ be an element of $A^*$,
and let $L,L_i$ be regular languages over $A$.
%$K$  a regular language over $B$,
%$L'$ a regular subset of $(A^*)^n$, and
%$p_i:(A^*)^n \ra A^*$ the projection map on the
%$i$-th coordinate.
Then the following languages are
also regular:
\begin{enumerate}
%\item (Homomorphic image) $\phi(L)$.
%\item (Homomorphic preimage) $\phi^{-1}(K)$.
\item (Quotient) $L_z:=\{w\in A^*: wz\in L\}$.
\item (Product) $L_1 \times L_2 \times \cdots \times L_n$.
%\item (Projection) $p_i(L)$.
\end{enumerate}
\end{lem}

The definition of a rewriting system is very close to that
of prefix-rewriting in Section~\ref{sec:intro}.
A {\em convergent rewriting system} for a group $G$
consists of a finite set $A$ and  a set
$R \subset A^* \times A^*$ of ordered pairs of words
over $A$ such that $G$ is presented, as a monoid, by
$
G = Mon \langle A \mid \{u=v \mid (u,v) \in R\} \rangle,
$
and the rewritings $wuz \ra wvz$ for all $(u,v) \in R$ and $w,z \in A^*$
satisfy the termination and confluence properties; that is,
there is no infinite sequence of rewritings, and  
the set of irreducible words 
is a set of (unique) normal forms for $G$.
(The only alteration to the definition of CP-RS needed 
to obtain the definition of convergent rewriting system
is in the prefix $w$ prepended in the rewritings.)
The rewriting system is {\em finite} if $R$ is a finite set,
and the system is {\em regular} 
if the subset $R \subset A^* \times A^*$ 
is (synchronously) regular.

%%%%%%%%%%%%%%%%%%%%%%%%%%%%%%%%%%%%%%%%%%%%%%%%%%%%%%%%%%%%%%%%%%%%%%%%
%%%%%%%%%%%%%%%%%%%%%%%%%%%%%%%%%%%%%%%%%%%%%%%%%%%%%%%%%%%%%%%%%%%%%%%%

%\subsection{Rewriting and prefix-rewriting systems}\label{subsec:rs}

%%%%%%%%%%%%%%%%%%%%%%%%%%%%%%%%%%%%%%%%%%%%%%%%%%%%%%%%%%%%%%%%%%%%%%%%
%%%%%%%%%%%%%%%%%%%%%%%%%%%%%%%%%%%%%%%%%%%%%%%%%%%%%%%%%%%%%%%%%%%%%%%%

%A {\em derivation} of a word $w$ over the alphabet $A$ is a
%sequence of successive applications of (prefix)-rewriting rules from $R$.
%A finite derivation, that terminates at a word $w'$, is denoted
%$w \ra^* w'$.  In the case that $w \neq w'$ and so the derivation includes the application
%of at least one rewriting, the derivation is denoted $w \ra^+ w'$.

%For a convergent (prefix-)rewriting system $R$ over an alphabet $A$, 
%every word $w \in A^*$ can be transformed by applying a sequence of rewriting rules, 
%to the unique irreducible word, or normal form, $\nf{w}$ of $w$. 

%{\color{red}
%****************************************************************

%Write/rewrite this section.

%****************************************************************}

%%%%%%%%%%%%%%%%%%%%%%%%%%%%%%%%%%%%%%%%%%%%%%%%%%%%%%%%%%%%%%%%%%%%%%%%%%%%%
%%%%%%%%%%%%%%%%%%%%%%%%%%%%%%%%%%%%%%%%%%%%%%%%%%%%%%%%%%%%%%%%%%%%%%%%%%%%%

\subsection{The Guba-Sapir 
rewriting system for $F$ and its normal forms}\label{subsec:gsrs}

%%%%%%%%%%%%%%%%%%%%%%%%%%%%%%%%%%%%%%%%%%%%%%%%%%%%%%%%%%%%%%%%%%%%%%%%%%%%%
%%%%%%%%%%%%%%%%%%%%%%%%%%%%%%%%%%%%%%%%%%%%%%%%%%%%%%%%%%%%%%%%%%%%%%%%%%%%%

Guba and Sapir  state in \cite[Theorem~2]{GuSa} that Thompson's group $F$ admits a 
convergent rewriting system and a regular set of normal forms in the 
standard generators $x=x_0$ and $y=x_1$.  
The convergent rewriting system given in \cite{GuSa} is: 
%as follows: 
%$\Sigma$ contains the ``free cancellation rules'' 
%$x^\epsilon x^{-\epsilon}\rightarrow 1$ and $y^\epsilon y^{-\epsilon}\rightarrow 1$ 
%for $\epsilon\in\{1,-1\}$. In addition we have the following rewriting rules 
%for every $i\in\mathbb N$, $\epsilon\in\{1,-1\}$.
$$
\Sigma = \left\{ 
\begin{array}{lll|l}
aa^{-1} &\ra& \emptyword,  & a \in \{x^{\pm 1},y^{\pm 1}\},\\ 
y^\epsilon x^iy &\rightarrow& x^iyx^{-i-1}y^{\epsilon}x^{i+1},&
\epsilon\in\{1,-1\} \text{ and } i \geq 1 \\
y^\epsilon x^{i+1}y^{-1} &\rightarrow& x^{i+1}y^{-1}x^{-i}y^{\epsilon}x^i & 
\end{array} 
 \right\}.
$$
We refer to the $\Sigma$-rewriting rules $aa^{-1} \ra \emptyword$ 
as free reductions, as usual, and to 
\[
y^\epsilon x^iy \rightarrow x^iyx^{-i-1}y^{\epsilon}x^{i+1}
\]
and 
\[
y^\epsilon x^{i+1}y^{-1} \rightarrow x^{i+1}y^{-1}x^{-i}y^{\epsilon}x^{i},  
\]
with $\epsilon\in\{1,-1\}$ and $i \geq 1$, as the 
{\em $y$-rule} and {\em $y^{-1}$-rule of size} $i$, respectively.

Throughout the rest of the paper, the symbol $\Nfgs$ will denote 
the set of normal forms 
over the generating set $A$ defined by the rewriting system $\Sigma$; 
that is, $\Nfgs$ is the 
set of words that cannot be rewritten using $\Sigma$. 
%Also throughout the rest of the paper, the symbol $T$ will denote
%the maximal tree $T$ in the Cayley graph $\Gamma$ consisting of
%the set of edges that lie on a path labeled by a word in $\Nfgs$
%and  starting at the vertex $1$. 
We note that the set $\Nfgs$ is regular, since
 $\Nfgs$ is the set of all words over $A$ which do not contain the 
left hand side of any of the rules in $\Sigma$ as a subword, and since 
the set %$\mathsf{LHS}(\Sigma)$ 
of left hand sides of rules in $\Sigma$  is a regular language. 
%then the language $\Nfgs=A^* \setminus A^*\mathsf{LHS}(\Sigma)A^*$ is also regular.
%Although $\Nfgs$ is a regular language, 
However, the appearance of the integer $i$ three
times on the right hand side of the $y$- and $y^{-1}$-rules of the 
rewriting system $\Sigma$
shows that
the set of padded words obtained from $\Sigma$ does not obey the Pumping Lemma for
regular languages~\cite[Lemma~3.1]{hu}, nor the Pumping Lemma for 
context-free languages~\cite[Lemma~6.1]{hu}, and so the
rewriting system~$\Sigma$ is neither (synchronously) regular nor context-free.  

The following lemma is immediate from the description of $\Sigma$. 

\begin{lem}\label{lem:normal}
	The set of normal forms $\Nfgs$ is the set of all words over $A$ of the form 
\[
 x^{i_n}y^{\epsilon_n} \cdots x^{i_1}y^{\epsilon_1}x^{i_0},
\]
	such that $n \ge 0$, each $\epsilon_j \in \{\pm 1\}$, each $i_j\in\mathbb{Z}$, and for 
        every $j\in\{1,\dots,n-1\}$ the following hold. 
	\begin{enumerate}
		\item[(1)] If $i_j=0$ then $\epsilon_j=\epsilon_{j+1}$ 
		(i.e., the word is freely reduced).
		\item[(2)] If $\epsilon_{j}=1$ then $i_j\le 0$.  
		\item[(3)] If $\epsilon_{j}=-1$ then $i_j\le 1$.  
	\end{enumerate}
\end{lem}

For any element $g \in F$, write $\nf{g}$ for the normal form
in $\Nfgs$ representing $G$; similarly, for any word $w$ over $A$,
write $\nf{w}$ for the unique word in $\Nfgs$ such that $w \equalinf \nf{w}$.

Throughout this paper we also let $T$ denote the subtree of $\Gamma$
associated with $\Nfgs$; in other words, $T$ is the subtree of $\Gamma$ such that the set of
labels of nonbacktracking paths in $T$, starting from the identity vertex, is the set of normal
forms in $\Nfgs$.

Next we give a description of the edges that do not lie in the tree $T$.  
Note that a directed edge $e=e_{\gamma,a}$ with $\gamma=\nf{e_-}$ lies on the tree $T$ 
if and only if either $\gamma a\in\Nfgs$ or the last letter of $\gamma$ is $a^{-1}$. 
%Note that if $za$ is non reduced 
%then $z$ terminates with $a^{-1}$ so that the freely reduced word freely equivalent to  $za$ 
%is a prefix of $z$ and as such, belongs to $\Nfgs$. 
That is, the edge $e=e_{\gamma,a}$ lies on $T$ 
if and only if the freely reduced word freely equivalent to $\gamma a$ belongs to $\Nfgs$. 
Hence Lemma~\ref{lem:normal} implies the following.

\begin{lem}\label{lem:yyinv1}
	Let $e=e_{\gamma,a}$ be a directed edge of $\Gamma$ where $a\in\{x^{\pm 1},y^{\pm 1}\}$ and 
	$$\gamma \equalword x^{i_n}y^{\epsilon_n}\cdots x^{i_1}y^{\epsilon_1}x^{i_0}\in \mathcal N.$$
	 Then $e$ does not lie on the tree $T$ if and only if one of the following holds. 
	 \begin{enumerate}
	 	\item[(1)] $\lab(e)\equalword  y$, $n \geq 1$, and $i_0\ge 1$.
	 	\item[(2)] $\lab(e)\equalword  y^{-1}$, $n \geq 1$,  and $i_0\ge 2$.
	 \end{enumerate}
	 In particular, all edges of $\Gamma$ labeled by $x^{\pm 1}$ lie on the tree $T$. 
\end{lem}

We end this section by defining a few parameters associated to every word in 
normal form. 

Let $\gamma= x^{i_n}y^{\epsilon_n} \cdots x^{i_1}y^{\epsilon_1}x^{i_0}$ be a 
word in normal form. Denote by $s(\gamma)$ the vector $s(\gamma) = (s_n,\dots,s_0)$ 
of cumulative $x$-exponents in $\gamma$ defined, for $k=0,\dots,n$, by 
\[ 
s_k := i_k+i_{k-1}+ \dots + i_0.  
\]
Note that the vector has at least one entry, 
even when the word $\gamma$ is empty (indeed, in that case, $s(\gamma)=(0)$). 
We also define two cutoff points, namely, let 
\[ 
m(\gamma) := \min\{k \mid s_k \leq 0\},
\]
if the set on the right hand side is nonempty; otherwise, let $m(\gamma):=n$, and similarly, 
\[ 
m'(\gamma) := \min\{k \mid s_k \leq 1\},
\]
if the set on the right hand side is nonempty; otherwise, let $m'(\gamma):=n$. 
We extend our notation to elements $g$ of $F$ by declaring 
$s(g)$ to mean $s(\nf{g})$, $m(g)$ to mean $m(\nf{g})$, and so on.

%%%%%%%%%%%%%%%%%%%%%%%%%%%%%%%%%%%%%%%%%%%%%%%%%%%%%%%%%%%%%%%%%%%%%%%%%%%%%
%%%%%%%%%%%%%%%%%%%%%%%%%%%%%%%%%%%%%%%%%%%%%%%%%%%%%%%%%%%%%%%%%%%%%%%%%%%%%

\section{Normal forms of the endpoints of edges outside of the 
normal form tree}\label{sec:prelim}

%%%%%%%%%%%%%%%%%%%%%%%%%%%%%%%%%%%%%%%%%%%%%%%%%%%%%%%%%%%%%%%%%%%%%%%%%%%%%
%%%%%%%%%%%%%%%%%%%%%%%%%%%%%%%%%%%%%%%%%%%%%%%%%%%%%%%%%%%%%%%%%%%%%%%%%%%%%

%-------------------------------------------------------------
%\subsection{An edge $e$ outside of the normal form tree}

We now turn to the analysis of the relation between the normal forms of the endpoints of edges that are not on the 
normal form tree $T$. It will be convenient to set up and fix a situation and notation used throughout
Sections~\ref{sec:prelim} and~\ref{sec:order}.
%the rest of the paper.

Let $e$ be an edge not on the tree $T$, labeled by $y$, connecting $g$ to $g'$;
that is, 
\[
e: \stackrel{\hbox{ $g$ }}{\bullet} \stackrel{y}{\longrightarrow} \stackrel{\hbox{$g'$}}{\bullet}
\]
satisfies
$e_-=g$, $e_+=g'$, and $g' =_F gy$.
We also consider the inverse edge $e^{-1}$, labeled by $y^{-1}$, 
directed from $g'$ to $g$. 
 Let the normal forms of the endpoints of $e$ be given by 
\[ 
 \gamma := \nf{g} = x^{i_n}y^{\epsilon_n}\cdots x^{i_1}y^{\epsilon_1}x^{i_0},
\]
with each  $\epsilon_\ell \in \{\pm 1\}$, and 
\[ 
\gamma' := \nf{g'} = x^{j_{n'}}y^{\varepsilon_{n'}}\cdots x^{j_1}y^{\varepsilon_1}x^{j_0},
\]
with each  $\varepsilon_\ell \in \{\pm 1\}$. 
Note that, along with the usual requirements from Lemma~\ref{lem:normal} that the normal forms 
$\gamma$ and $\gamma'$ must satisfy, we also have the additional requirements from Lemma~\ref{lem:yyinv1}, 
since $e$ is not on the tree $T$. Thus, $n,n' \geq 1$, $i_0 \geq 1$ and $j_0 \geq 2$. 

Using the notation from Section~\ref{subsec:gsrs}, let $m:=m(\gamma)$ and 
$m':= m'(\gamma')$.  Then 
\[ 
 s(\gamma) = (s_n,\dots,s_0), \qquad s(\gamma') = (s'_{n'},\dots,s'_0),
\]
\[
 m = m(\gamma)=\min\{k\mid s_k \leq 0\}, \qquad m' = m'(\gamma') = \min\{k \mid s'_k \leq 1\}.
\]

Since the rewriting system $\Sigma$ is convergent, there is a sequence of 
$\Sigma$-rewritings from any word in $A^*$ to its normal form; indeed, 
there may be many different such derivations.
We define the {\em standard $\Sigma$-rewriting}, or
{\em standard $\Sigma$-derivation}, of a word in $A^*$ to be 
the sequence of $\Sigma$-rewritings to the normal form 
where at each rewriting step a $\Sigma$-rule is applied to the 
shortest possible rewritable prefix. 
%that applies a $\Sigma$-rule to the shortest possible rewritable prefix 
%at each rewriting step.
%%the word $\gamma y$ to $\gamma'$.
In the following subsections we provide details of the
standard $\Sigma$-rewritings of $\gamma y$ and $\gamma' y^{-1}$ to their 
normal forms $\gamma'$ and $\gamma$, respectively.
We also extract some parameters associated to those standard $\Sigma$-rewritings 
directly from $\gamma$ and $\gamma'$, and more specifically, 
from $s(\gamma)$, $s(\gamma')$, $m$, and $m'$.

%-------------------------------------------------------------

%%%%%%%%%%%%%%%%%%%%%%%%%%%%%%%%%%%%%%%%%%%%%%%%%%%%%%%%%%%%%%%%%%%%%%%%%%%%%
%%%%%%%%%%%%%%%%%%%%%%%%%%%%%%%%%%%%%%%%%%%%%%%%%%%%%%%%%%%%%%%%%%%%%%%%%%%%%

\subsection{Rewriting $\gamma y$ to $\gamma'$}\label{subsec:multbyy}

%%%%%%%%%%%%%%%%%%%%%%%%%%%%%%%%%%%%%%%%%%%%%%%%%%%%%%%%%%%%%%%%%%%%%%%%%%%%%
%%%%%%%%%%%%%%%%%%%%%%%%%%%%%%%%%%%%%%%%%%%%%%%%%%%%%%%%%%%%%%%%%%%%%%%%%%%%%

In this section we describe how to obtain the normal form  $\gamma'$ for the
terminal vertex $e_+ \equalinf \gamma y$ of $e=e_{\gamma,y}$, starting from 
the word $\gamma y$.

\begin{lem}\label{lem:multbyy} 
We have $m \geq 1$, and the following assertions hold. 
\begin{itemize}
\item[(1)]  
	If either $m=n$ or $s_m \neq 0$ or $y^{\epsilon_{m+1}} =y$, then 
	\[ 
	\gamma' =\nf{\gamma y}
	= x^{i_n}y^{\epsilon_n}\cdots x^{i_{m+1}}y^{\epsilon_{m+1}}
	x^{s_m} y x^{-s_{m-1}-1}y^{\epsilon_{m}} 
	x^{i_{m-1}}y^{\epsilon_{m-1}} \cdots x^{i_{1}}y^{\epsilon_1}x^{i_0+1}.
	\]

\item[(2)]
	Otherwise, if $m < n$, $s_m = 0$, and $y^{\epsilon_{m+1}} =y^{-1}$, then
	\[
	\gamma'=\nf{\gamma y}
	= x^{i_n}y^{\epsilon_n}\cdots x^{i_{m+2}} 
	y^{\epsilon_{m+2}}x^{i_{m+1}+i_{m}-1} y^{\epsilon_{m}} \\ x^{i_{m-1}}y^{\epsilon_{m-1}} \cdots x^{i_{1}}y^{\epsilon_1}x^{i_0+1}.
	\]
\end{itemize}
In each case, there is a sequence of $\Sigma$-rewriting rules leading from $\gamma y$ to its normal form $\gamma'$ such that no $y^{-1}$-rules are ever used, exactly $m$ $y$-rules are used and their sizes are, in the order in which they are used,
\[ 
s_0,~s_1,~s_2,~\dots,~s_{m-1}. 
\]
\end{lem}

\begin{proof}
Since $s_0=i_0 \geq 1$, we have $m \geq 1$. 

Our standard $\Sigma$-derivation proceeds from $\gamma y$ in $m$ phases. 
During each of the $m$ phases exactly one $y$-rule is used followed by 
a complete free reduction. In detail, phases 1 through $m-1$ proceed as follows. 
In phase 1 the $y$-rule of size $s_0$ is applied to $y^{\epsilon_1}x^{s_0}y$, 
followed by a complete free reduction of the obtained word. 
Then, in phase 2 the $y$-rule of size $s_1$ is applied to $y^{\epsilon_2}x^{s_1}y$, 
followed by a complete free reduction, and so on, 
until in phase $m-1$ the $y$-rule of size $s_{m-2}$ is applied to 
$y^{\epsilon_{m-1}}x^{s_{m-2}}y$, followed by a complete free reduction. 
In phase $m$, the $y$-rule of size $s_{m-1}$ is applied to 
$y^{\epsilon_m}x^{s_{m-1}}y$, followed by free reduction (not necessarily complete), 
yielding the word 
\begin{equation}\label{e:cutoff}
x^{s_m} y x^{-s_{m-1}-1} y^{\epsilon_m} x^{i_{m-1}} y^{\epsilon_{m-1}} \dots x^{i_1} y^{\epsilon_1} x^{i_0+1}, 
\end{equation}
if $n=m$ or, otherwise, if $n>m$, 
\begin{equation}\label{e:nocutoff}
x^{i_n}y^{\epsilon_n} \dots \underline{x^{i_{m+1}} y^{\epsilon_{m+1}} x^{s_m} y x^{-s_{m-1}-1}} y^{\epsilon_m} x^{i_{m-1}} y^{\epsilon_{m-1}} \dots x^{i_1} y^{\epsilon_1} x^{i_0+1}.  
\end{equation}
At this point, to finish phase $m$ of the standard $\Sigma$-rewriting to $\gamma'$ 
we just need to complete the free reduction. 

Let us indicate why the rewriting described up to this point 
(before this last free reduction) is actually possible; that is, 
why we eventually arrive, during phase $m$, at a word of 
the form~\eqref{e:cutoff} or~\eqref{e:nocutoff} in the appropriate cases. 

The original word 
\[ 
\gamma y = x^{i_n}y^{\epsilon_n}\dots x^{i_1}y^{\epsilon_1}x^{i_0}y
 = x^{i_n}y^{\epsilon_n}\dots x^{i_1}\underline{y^{\epsilon_1}x^{s_0}y}
\] 
is freely reduced and since its prefix $\gamma$ is a normal form, 
the only available $\Sigma$-rewriting is the application of the 
$y$-rule of size $s_0$ to the underlined portion of the word, 
yielding a word that freely reduces to 
\[
x^{s_1} y x^{-s_0-1} y^{\epsilon_1} x^{i_0+1}, 
\]
if $n=m=1$, and to 
\[
x^{i_n}y^{\epsilon_n} \dots x^{i_2}\underline{y^{\epsilon_{2}}x^{s_1} y} x^{-s_0-1} y^{\epsilon_1} x^{i_0+1}, 
\]
in all other cases. 
If $m=1$ we are done (we arrived at a word of the form~\eqref{e:cutoff} 
or~\eqref{e:nocutoff} during phase 1). 

If $n \geq m \geq 2$, then we are in the second case above. 
The prefix $y^{\epsilon_n} \dots y^{\epsilon_2}$ is nonempty 
and $y^{\epsilon_2}$ actually appears in it. 
Note also that $s_0 \geq 1$, which implies that $-s_0-1 \leq -2$ and, 
since $m \geq 2$, $s_1 \geq 1$. 
Thus the word is freely reduced and the only available $\Sigma$-rewriting 
is the application of the $y$-rule of size $s_1$ to the underlined portion 
of the word, yielding a word that freely reduces to 
\[
x^{s_2} y x^{-s_1-1} y^{\epsilon_2} x^{i_1} y^{\epsilon_1} x^{i_0+1}, 
\]
if $n=m=2$, and to 
\[
x^{i_n}y^{\epsilon_n} \dots x^{i_3}\underline{y^{\epsilon_{3}}x^{s_2} y} x^{-s_1-1} y^{\epsilon_2} x^{i_1} y^{\epsilon_1} x^{i_0+1}, 
\]
in all other cases. 
If $m=2$ we are done (we arrived at a word of the form~\eqref{e:cutoff} 
or~\eqref{e:nocutoff} during phase 2).

If $n \geq m \geq 3$, then we are in the second case above. The portion of the word $y^{\epsilon_n} \dots y^{\epsilon_3}$ is nonempty and $y^{\epsilon_3}$ actually appears in it. Note also that $s_1 \geq 1$, which implies that $-s_1-1 \leq -2$ and, since $m \geq 3$, $s_2 \geq 1$. Thus the word is freely reduced and the only available $\Sigma$-rewriting is the application of the $y$-rule of size $s_2$ to the underlined portion of the word...
Et cetera; we continue this process through the $m$ phases.

Let us consider now the words~\eqref{e:cutoff} or~\eqref{e:nocutoff} obtained during phase $m$. 

Note that, since $s_{m-1} \geq 1$, we have $-s_{m-1}-1 \leq -2$. 
Therefore, if $n=m$, the word~\eqref{e:cutoff} is freely reduced and, 
since no $y^\pm$ rules are applicable, it is the normal form of $\gamma y$. 
In the other case, $n>m$, we obtain the word~\eqref{e:nocutoff}. 
If $s_m \neq 0$ or $y^{\epsilon_{m+1}}=y$ this word is also freely reduced 
and it is the normal form of $\gamma y$. 
Finally, if $n>m$, $s_m=0$, and $y^{\epsilon_{m+1}} = y^{-1}$, then there are 
further free reductions in the underlined portion of the word~\eqref{e:nocutoff}, 
yielding the word 
\[
x^{i_n}y^{\epsilon_n} \dots x^{i_{m+2}}y^{\epsilon_{m+2}} x^{i_{m+1}-s_{m-1}-1} y^{\epsilon_m} x^{i_{m-1}} y^{\epsilon_{m-1}} \dots x^{i_1} y^{\epsilon_1} x^{i_0+1}.  
\]	
We claim that, in this last case, either $n=m+1$ or $i_{m+1}-s_{m-1} -1 \leq -1$, 
which implies that the obtained word is freely reduced and, 
since no $y^\pm$-rules can be applied, it is the normal form of $\gamma y$. 
Indeed, if $n>m+1$ then $-s_{m-1} \leq -1$ (since $s_{m-1} \geq 1$) 
and $i_{m+1} \leq 1$, which yields $i_{m+1}-s_{m-1} -1 \leq -1$.
Since $s_m=0$, then $-s_{m-1}=i_m$, completing the proof of
part (2) of the lemma.
\end{proof}

\begin{rmk}
It will be useful to record some of the conclusions of Lemma~\ref{lem:multbyy} in a slightly different form. Namely, if either $m=n$ or $s_m \neq 0$ or $y^{\epsilon_{m+1}} =y$, then there are no free cancellations of $y$ letters during the standard $\Sigma$-rewriting, $n'=n+1$, and the $x$-exponents and the cumulative $x$-exponents of $\gamma'$ are given in the left half of Table~\ref{t:exponentsy}.
\begin{table}[!ht]
\begin{tabular}{cc}
	$
	\begin{array}{|c|c|c|}
	\hline 
	k & j_k & s'_k \\
	\hline 
	0  & i_0+1    & s_0+1 \\
	1  & i_1      & s_1+1 \\
	\vdots & \vdots & \vdots \\
	m-1 & i_{m-1} & s_{m-1}+1 \\
	m   & -s_{m-1}-1 & 0 \\
	m+1 & s_m  & s_{m} \\
	m+2 & i_{m+1} & s_{m+1} \\
	\vdots & \vdots & \vdots \\
	\vdots & \vdots & \vdots \\
	n+1 & i_{n} & s_{n}\\
	\hline
	\end{array}
	$
	&
	$
	\begin{array}{|c|c|c|}
	\hline 
	k & j_k & s'_k \\
	\hline 
	0  & i_0+1    & s_0+1 \\
	1  & i_1      & s_1+1 \\
	\vdots & \vdots & \vdots \\
	m-1 & i_{m-1} & s_{m-1}+1 \\
	m   & i_{m+1}+i_m-1 & s_{m+1} \\
	m+1 & i_{m+2} & s_{m+2} \\
	\vdots & \vdots & \vdots \\
	n-1 & i_{n} & s_{n}\\
	\hline
	\end{array}
	$
\end{tabular}
\caption{$x$-exponents and cumulative $x$-exponents of $\gamma'$: the no $y$-cancellation case (left) and the $y$-cancellation case (right)}
\label{t:exponentsy}
\end{table}
We call this case the {\em no $y$-cancellation case}. Otherwise, if $m<n$, $s_m=0$, and $y^{\epsilon_{m+1}} =y$, then there is a single free cancellation of $y^\pm$-letters during the standard $\Sigma$-rewriting, $n'=n-1$ (note that, in this case $n \geq m+1 \geq 2$), and the $x$-exponents and the cumulative $x$-exponents of $\gamma'$ are given in the right half of Table~\ref{t:exponentsy}. We call this latter case the {\em $y$-cancellation case}. 
\end{rmk}

%-------------------------------------------------------------

%%%%%%%%%%%%%%%%%%%%%%%%%%%%%%%%%%%%%%%%%%%%%%%%%%%%%%%%%%%%%%%%%%%%%%%%%%%%%
%%%%%%%%%%%%%%%%%%%%%%%%%%%%%%%%%%%%%%%%%%%%%%%%%%%%%%%%%%%%%%%%%%%%%%%%%%%%%

\subsection{Rewriting $\gamma' y^{-1}$ to $\gamma$}\label{subsec:multbyyy}

%%%%%%%%%%%%%%%%%%%%%%%%%%%%%%%%%%%%%%%%%%%%%%%%%%%%%%%%%%%%%%%%%%%%%%%%%%%%%
%%%%%%%%%%%%%%%%%%%%%%%%%%%%%%%%%%%%%%%%%%%%%%%%%%%%%%%%%%%%%%%%%%%%%%%%%%%%%

The previous lemma describes, in much detail, how to ``multiply by $y$''. The next lemma, regarding ``multiplication by $y^{-1}$'', has a similar statement and an analogous proof, so we skip some but not all details, since there are a few subtle points that are sufficiently different. 

\begin{lem}\label{lem:multbyyy}
We have $m' \geq 1$, and the following assertions hold.
\begin{itemize}
\item[(1)]  
	If either $m'=n'$ or $s'_{m'} \neq 0$ or $y^{\varepsilon_{m'+1}} =y^{-1}$, then 
	\[
	\gamma = \nf{\gamma' y^{-1}}
	= x^{j_{n'}}y^{\varepsilon_{n'}}\cdots x^{j_{m'+1}}y^{\varepsilon_{m'+1}}
	x^{s'_{m'}} y^{-1} x^{-s'_{m'-1}+1}y^{\varepsilon_{m'}} 
	x^{j_{m'-1}}y^{\varepsilon_{m'-1}} \cdots x^{j_{1}}y^{\varepsilon_1}x^{j_0-1}.
	\]

\item[(2)]
	Otherwise, if $m' < n'$, $s'_{m'} = 0$, and $y^{\varepsilon_{m'+1}} =y$, then
	\[
	\gamma = \nf{\gamma' y^{-1}}
	= x^{j_{n'}}y^{\varepsilon_{n'}}\cdots x^{j_{m'+2}} 
	y^{\varepsilon_{m'+2}}x^{j_{m'+1}+j_{m'}+1} y^{\varepsilon_{m'}} \\ x^{j_{m'-1}}y^{\varepsilon_{m'-1}} \cdots x^{j_{1}}y^{\varepsilon_1}x^{j_0-1}.
	\]
\end{itemize}	
In each case, there is a sequence of $\Sigma$-rewriting rules leading from $\gamma' y^{-1}$ to its normal form $\gamma$ such that no $y$-rules are ever used, exactly $m'$ $y^{-1}$-rules are used and their sizes are, in the order in which they are used,
\[ 
s'_0-1,~s'_1-1,~s'_2-1,~\dots~,~s'_{m'-1}-1. 
\]
\end{lem}

\begin{proof}
Since $s'_0=j_0 \geq 2$, we have $m' \geq 1$. 

Our standard $\Sigma$-derivation proceeds from $\gamma' y^{-1}$ in $m'$ phases. 
During each of the $m'$ phases exactly one $y^{-1}$-rule is used followed by 
a complete free reduction. In detail, phases 1 through $m'-1$ proceed as follows. 
In phase 1 the $y^{-1}$-rule of size $s'_0-1$ is applied to 
$y^{\varepsilon_1}x^{s'_0}y^{-1}$, followed by a complete free reduction of 
the obtained word. Then, in phase 2 the $y^{-1}$-rule of size $s'_1-1$ is 
applied to $y^{\varepsilon_2}x^{s'_1}y^{-1}$, 
followed by a complete free reduction, and so on, until 
in phase $m'-1$ the $y^{-1}$-rule of size $s'_{m'-2}-1$ is applied to 
$y^{\varepsilon_{m'-1}}x^{s'_{m-2}}y^{-1}$, followed by a complete free reduction. 
In phase $m'$, the $y^{-1}$-rule of size $s'_{m'-1}-1$ is applied to 
$y^{\varepsilon_{m'}}x^{s'_{m'-1}}y^{-1}$, followed by free reduction 
(not necessarily complete), yielding the word 
\begin{equation}\label{e:cutoff2}
x^{s'_{m'}} y^{-1} x^{-s'_{m'-1}+1} y^{\varepsilon_{m'}} x^{j_{m'-1}} y^{\varepsilon_{m'-1}} \dots x^{j_1} y^{\varepsilon_1} x^{j_0-1}, 
\end{equation}
if $n'=m'$ or, otherwise, the word 
\begin{equation}\label{e:nocutoff2}
x^{j_{n'}}y^{\varepsilon_{n'}} \dots \underline{x^{j_{m'+1}} y^{\varepsilon_{m'+1}} x^{s'_{m'}} y^{-1} x^{-s'_{m'-1}+1}} y^{\varepsilon_{m'}} x^{j_{m'-1}} y^{\varepsilon_{m'-1}} \dots x^{j_1} y^{\varepsilon_1} x^{j_0-1}.  
\end{equation}

Note that, since $s'_{m'-1} \geq 2$, we have $-s'_{m'-1}+1 \leq -1$. 
Therefore, if $n'=m'$, the word~\eqref{e:cutoff2} is freely reduced and, 
since no $y^\pm$ rules are applicable, it is the normal form of $\gamma' y^{-1}$. 
In the other case, $n'>m'$, we obtain the word~\eqref{e:nocutoff2}. 
If $s'_{m'} \neq 0$ or $y^{\varepsilon_{m'+1}}=y^{-1}$ this word is also 
freely reduced and it is the normal form of $\gamma' y^{-1}$. 
Finally, if $n'>m'$, $s'_{m'}=0$, and $y^{\varepsilon_{m'+1}} = y$, 
then there are further free reductions in the underlined portion 
of the word~\eqref{e:nocutoff2}, yielding the word 
\[
x^{j_{n'}}y^{\varepsilon_{n'}} \dots x^{j_{m'+2}}y^{\varepsilon_{m'+2}} 
   x^{j_{m'+1}-s'_{m'-1}+1} y^{\varepsilon_{m'}} x^{j_{m'-1}} y^{\varepsilon_{m'-1}} 
    \dots x^{j_1} y^{\varepsilon_1} x^{j_0-1}.  
\]	
We claim that, in this last case, either $n'=m'+1$ or $j_{m'+1}-s'_{m'-1} +1 \leq -1$, 
which implies that the obtained word is freely reduced and, since no 
$y^\pm$-rules can be applied, it is the normal form of $\gamma' y^{-1}$. 
Indeed, if $n'>m'+1$ then $s'_{m'-1} \leq -2$ 
(since $s'_{m'-1} \geq 2$) and $j_{m'+1} \leq 0$ 
(since $\varepsilon_{m'+1} =1$), which yields $j_{m'+1}-s'_{m'-1} +1 \leq -1$.
Since $s'_{m'}=0$, then $-s'_{m'-1}=j_{m'}$, completing the proof of
part (2) of the lemma.
\end{proof}

\begin{rmk}
It will also be useful to record some of the conclusions of Lemma~\ref{lem:multbyyy} in a slightly different form. Namely, if either $m'=n'$ or $s'_{m'} \neq 0$ or $y^{\epsilon_{m'+1}} =y$, then there are no free cancellations of $y$ letters during the standard $\Sigma$-rewriting, $n=n'+1$, and the $x$-exponents and the cumulative $x$-exponents of $\gamma$ are given in the left half of Table~\ref{t:exponentsyy}.
\begin{table}[!ht]
\begin{tabular}{cc}
	$
	\begin{array}{|c|c|c|}
	\hline 
	k      & i_k & s_k \\
	\hline 
	0      & j_0-1    & s'_0-1 \\
	1      & j_1      & s'_1-1 \\
	\vdots & \vdots & \vdots \\
	m'-1   & j_{m'-1} & s'_{m'-1}-1 \\
	m'     & -s'_{m'-1}+1 & 0 \\
	m'+1   & s'_{m'}  & s'_{m'} \\
	m'+2   & j_{m'+1} & s'_{m'+1} \\
	\vdots & \vdots & \vdots \\
	\vdots & \vdots & \vdots \\
	n'+1   & j_{n'} & s'_{n'}\\
	\hline
	\end{array}
	$
	&
	$
	\begin{array}{|c|c|c|}
	\hline 
	k & i_k & s_k \\
	\hline 
	0  & j_0-1    & s'_0-1 \\
	1  & j_1      & s'_1-1 \\
	\vdots & \vdots & \vdots \\
	m'-1 & j_{m'-1} & s'_{m'-1}-1 \\
	m'   & j_{m'+1}+j_{m'}+1 & s'_{m'+1} \\
	m'+1 & j_{m'+2} & s'_{m'+2} \\
	\vdots & \vdots & \vdots \\
	n'-1 & j_{n'} & s'_{n'}\\
	\hline
	\end{array}
	$
\end{tabular}
\caption{$x$-exponents and cumulative $x$-exponents of $\gamma'$: the no $y^{-1}$-cancellation case (left) and the $y^{-1}$-cancellation case (right)}
\label{t:exponentsyy}
\end{table}
We call this case the {\em no $y^{-1}$-cancellation case}. Otherwise, if $m'<n'$, $s'_{m'}=0$, and $y^{\epsilon_{m'+1}} =y$, then there is a single free cancellation of $y^\pm$-letters during the standard $\Sigma$-rewriting, $n=n'-1$, and the $x$-exponents and the cumulative $x$-exponents of $\gamma$ are given in the right half of Table~\ref{t:exponentsyy}. We call this latter case the {\em $y^{-1}$-cancellation case}. 
\end{rmk}

In the Lemmas~\ref{lem:multbyy} and~\ref{lem:multbyyy}, we produced standard $\Sigma$-rewritings of $\gamma y$ and $\gamma' y^{-1}$ when the edges corresponding to $y$ and $y^{-1}$, respectively, are not on the normal form tree. We do not need the following observation, but a careful reading of the proofs reveals that any $\Sigma$-rewriting of $\gamma y$ uses $m$ $y$-rules of the  same sizes and in the same order as in the standard rewriting, and uses no $y^{-1}$-rules. Similarly, any $\Sigma$-rewriting of $\gamma' y^{-1}$ uses $m'$ $y^{-1}$-rules of the same sizes and in the same order as in the standard rewriting, and uses no $y$-rules. In other words, the only freedom of choice we have during the $\Sigma$-rewriting of $\gamma y$ and $\gamma' y^{-1}$ is in the order of performing the free reductions (including the possibility of postponing some free reductions and applying some $y^\pm$-rules before the word is freely reduced). 

%%%%%%%%%%%%%%%%%%%%%%%%%%%%%%%%%%%%%%%%%%%%%%%%%%%%%%%%%%%%%%%%%%%%%%%%%%%%%
%%%%%%%%%%%%%%%%%%%%%%%%%%%%%%%%%%%%%%%%%%%%%%%%%%%%%%%%%%%%%%%%%%%%%%%%%%%%%

\section{Ordering the directed edges in the Cayley graph}\label{sec:order}

%%%%%%%%%%%%%%%%%%%%%%%%%%%%%%%%%%%%%%%%%%%%%%%%%%%%%%%%%%%%%%%%%%%%%%%%%%%%%
%%%%%%%%%%%%%%%%%%%%%%%%%%%%%%%%%%%%%%%%%%%%%%%%%%%%%%%%%%%%%%%%%%%%%%%%%%%%%

In this section we define a well-founded strict partial order on the
edges of the Cayley graph $\Gamma$ of $F$, which will be used in 
Section~\ref{sec:proof} in the proof of the termination of the flow function.

Throughout this section we use the same notation defined in
Section~\ref{sec:prelim}.  In particular, $e$ is a directed edge 
%of the Cayley graph $\Gamma=\Gamma_A(F)$ 
%of Thompson's group $F$ that does not
%lie in the tree $T$ determined by the set $\Nfgs$ of normal forms,
%the generator $y$ is the label of $e$, 
labeled by $y$, the initial vertex $g$
of $e$ has normal form $\gamma$ and the terminal vertex has 
normal form $\gamma'$,  
%\[
%\gamma := \nf{g} = x^{i_n}y^{\epsilon_n}\cdots x^{i_1}y^{\epsilon_1}x^{i_0},
%\]
%with each  $\epsilon_\ell \in \{\pm 1\}$, and 
%\[ 
%\gamma' := \nf{g'} = x^{j_{n'}}y^{\varepsilon_{n'}}\cdots x^{j_1}y^{\varepsilon_1}x^{j_0},
%\]
%with each  $\varepsilon_\ell \in \{\pm 1\}$. 
%Moreover,  
the vectors of cumulative $x$-exponents are
$ s(\gamma) = (s_n,\dots,s_0)$ and $s(\gamma') = (s'_{n'},\dots,s'_0)$,
and the cutoff points are $m:=m(\gamma)$ and $m':= m'(\gamma')$.

%%%%%%%%%%%%%%%%%%%%%%%%%%%%%%%%%%%%%%%%%%%%%%%%%%%%%%%%%%%%%%%%%%%%%%%%%%%%%
%%%%%%%%%%%%%%%%%%%%%%%%%%%%%%%%%%%%%%%%%%%%%%%%%%%%%%%%%%%%%%%%%%%%%%%%%%%%%

\subsection{Size sequence of an edge}\label{subsec:sizeseq}

%%%%%%%%%%%%%%%%%%%%%%%%%%%%%%%%%%%%%%%%%%%%%%%%%%%%%%%%%%%%%%%%%%%%%%%%%%%%%
%%%%%%%%%%%%%%%%%%%%%%%%%%%%%%%%%%%%%%%%%%%%%%%%%%%%%%%%%%%%%%%%%%%%%%%%%%%%%

In this subsection we define the size sequence of a directed edge 
$e$  of the
Cayley graph $\Gamma$ of Thompson's group $F$ that does not
lie in the tree $T$ determined by the set $\Nfgs$ of normal forms.

We begin by showing that, not surprisingly, the $\Sigma$-rewritings corresponding
 to pairs of inverse edges are closely related. 

\begin{lem}\label{l:same}
We have $m=m'$ and the sequence of $y$-rule sizes used in the rewriting of $\gamma y$ is the same (including the order) as the sequence of $y^{-1}$-rule sizes in the rewriting of $\gamma' y^{-1}$. 
\end{lem}

\begin{proof}
From Table~\ref{t:exponentsy}
% in the remark after Lemma~\ref{lem:multbyy}, 
we have $s'_k = s_k + 1 \geq 1+1 = 2$, for $k = 0,\dots,m-1$, and so $m' \ge m$. 
Similarly, from Table~\ref{t:exponentsyy}
we have $s_k = s'_k - 1 \geq 2-1 = 1$, for $k = 0,\dots,m'-1$, and so $m \ge m'$. 
Hence $m=m'$.
%Further, in the no $y$-cancellation case, we have $s'_{m} = 0 \leq 1$, 
%while in the $y$-cancellation case, if $n > m+1$ then we have 
%$s'_{m} = s_{m+1} = s_m +i_{m+1} \leq 0+1$. 
%In either case, 
We see from Lemma~\ref{lem:multbyyy} that the 
%number $m'$ of $y^{-1}$-rules used in the standard rewriting of 
%$\gamma'y^{-1}$ to $\gamma$ is $m$, and the 
sizes of the {$y^{-1}$-rules} used in the standard $\Sigma$-rewriting of 
$\gamma'y^{-1}$ to $\gamma$
are, for $k =0,\dots,m-1$ (thus, in order in which they are used), $s'_k -1 = s_k$, agreeing with the ordered list of sizes 
in the rewritings from $\gamma y$ to $\gamma'$ in Lemma~\ref{lem:multbyy}.
\end{proof}

We define the {\em size sequence} of the edge $e$ labeled by $y$ to be the sequence 
$$
\sigma(e)=(s_0,\dots,s_{m-1})
$$
of $y$-rule sizes used to ``cross'' the edge $e$, that is, 
to rewrite $\gamma y$ to $\gamma'$. Note that both the length $m(\gamma)$ 
of this sequence, denoted from now on also by $m(e)$, 
and its members can be easily ``read'' from the $x$-exponents of the 
normal form of the vertex $g=e_-$. Similarly, we define the size sequence 
of the edge $e^{-1}$ labeled by $y^{-1}$ to be the sequence 
$$
\sigma(e^{-1})=(s'_0-1,\dots,s'_{m'-1}-1)
$$ 
of $y^{-1}$-rule sizes used to ``cross'' the edge $e^{-1}$ 
that is, to rewrite $\gamma' y^{-1}$ to $\gamma$. 
Note that both the length $m'(\gamma')$ of this sequence, denoted from now 
on also by $m(e^{-1})$, and its members can be easily ``read'' 
from the $x$-exponents of the normal form of the vertex $g'=(e^{-1})_-$. 
Moreover, by Lemma~\ref{l:same}, these two sequences coincide, that is, 
they have the same length $m(e)=m(e^{-1})$ and exactly the same members 
$\sigma(e)=\sigma(e^{-1})$. Therefore, the size sequence of an edge can be 
easily read from the vertex of either end of the edge, by using either 
Table~\ref{t:exponentsy} or Table~\ref{t:exponentsyy}, 
as convenient and appropriate. 

%-------------------------------------------------------------

%%%%%%%%%%%%%%%%%%%%%%%%%%%%%%%%%%%%%%%%%%%%%%%%%%%%%%%%%%%%%%%%%%%%%%%%%%%%%
%%%%%%%%%%%%%%%%%%%%%%%%%%%%%%%%%%%%%%%%%%%%%%%%%%%%%%%%%%%%%%%%%%%%%%%%%%%%%

\subsection{Weight and order on edges outside of the tree of normal forms}\label{subsec:weight}

%%%%%%%%%%%%%%%%%%%%%%%%%%%%%%%%%%%%%%%%%%%%%%%%%%%%%%%%%%%%%%%%%%%%%%%%%%%%%
%%%%%%%%%%%%%%%%%%%%%%%%%%%%%%%%%%%%%%%%%%%%%%%%%%%%%%%%%%%%%%%%%%%%%%%%%%%%%

In this subsection we define a well-founded strict partial order on the directed edges of the Cayley graph $\Gamma$ of $F$ over the generating set $A = \{x^{\pm 1},y^{\pm 1}\}$. A well-founded strict partial order is a strict partial order for which there is no infinite descending chain;
such orders are useful in proving termination of (prefix-)rewriting systems.  

The order will be defined by simply using a weight function on the edges. The weight function uses the size sequence of the edge and the following sequence. 

Let $C$ be the sequence of positive integers defined recursively by 
\begin{align*}
 C(1) &= 1, \\
 C(2) &= 1, \\
 C(i) &= C(i-1) + 2C(i-2) + 2, \text{ for } i \geq 3.  
\end{align*}
The members of the sequence are explicitly given by 
\[
 C(i) = \frac{2}{3} \cdot 2^i - \frac{2}{3}(-1)^i  - 1,
\]
but the only important features of the sequence $C$ for our purposes are that it is nondecreasing sequence of positive integers such that 
%\begin{equation}\label{e:seq}
$C(i)>C(i-1)+C(1)$ for all $i \geq 3$.
%\end{equation}
Of course, formally speaking, this means that we could have chosen another, perhaps simpler, sequence than $C$ to define the weights, but we have chosen $C$ because it actually has meaningful  interpretation within our problem, described further in
Section~\ref{sec:interlude}.
%This is sequence A097074 in the OEIS.

\begin{defn}[Edge weight]
Let $e$ be a directed edge in $\Gamma$ that is not on the normal form tree $T$, with size sequence $\sigma(e)=(s_0,\dots,s_{m-1})$ of length $m=m(e)$. The {\em weight} $W(e)$ of the edge $e$ is the positive integer 
\[
 W(e) = \sum_{k=0}^{m-1} C(s_k). 
\]
\end{defn}

\begin{defn}[Edge order]
We define a strict partial order on the set of edges not on the normal from tree by 
\[
 e \prec e' \iff W(e) < W(e').
\]
\end{defn}

The relation $\prec$ defined above is indeed a well-founded strict partial order,
since the relation $<$ on the natural numbers is a well-founded strict partial order. %Moreover, it is well-founded, that is, there are no infinite descending chains 
%under $\prec$. 
Note also that Lemma~\ref{l:same} shows that
$W(e)=W(e^{-1})$ for any edge $e$.

%-------------------------------------------------------------

%%%%%%%%%%%%%%%%%%%%%%%%%%%%%%%%%%%%%%%%%%%%%%%%%%%%%%%%%%%%%%%%%%%%%%%%%%%%%
%%%%%%%%%%%%%%%%%%%%%%%%%%%%%%%%%%%%%%%%%%%%%%%%%%%%%%%%%%%%%%%%%%%%%%%%%%%%%

\section{Interlude: Motivation for the definition of the weight function}\label{sec:interlude}

%%%%%%%%%%%%%%%%%%%%%%%%%%%%%%%%%%%%%%%%%%%%%%%%%%%%%%%%%%%%%%%%%%%%%%%%%%%%%
%%%%%%%%%%%%%%%%%%%%%%%%%%%%%%%%%%%%%%%%%%%%%%%%%%%%%%%%%%%%%%%%%%%%%%%%%%%%%

In this section we give a pictorial, diagrammatic motivation for
the weight function defined in Section~\ref{subsec:weight}.
We note that the rest of the paper, including the proof of
Theorem~\ref{thm:fisautostk} in Section~\ref{sec:proof}, does
not require the material in this section; this section is provided to
explain how the proof came about.

%%%%%%%%%%%%%%%%%%%%%%%%%%%%%%%%%%%%%%%%%%%%%%%%%%%%%%%%%%%%%%%%%%%%%%%%%%%%%
%%%%%%%%%%%%%%%%%%%%%%%%%%%%%%%%%%%%%%%%%%%%%%%%%%%%%%%%%%%%%%%%%%%%%%%%%%%%%

\subsection{Background on van Kampen diagrams}\label{subsec:vkd}

%%%%%%%%%%%%%%%%%%%%%%%%%%%%%%%%%%%%%%%%%%%%%%%%%%%%%%%%%%%%%%%%%%%%%%%%%%%%%
%%%%%%%%%%%%%%%%%%%%%%%%%%%%%%%%%%%%%%%%%%%%%%%%%%%%%%%%%%%%%%%%%%%%%%%%%%%%%

Let ${\mathcal{P}} = \langle A \mid R \rangle$ be a presentation
for a group $G$, with the generating set $A$ closed under inversion.
Let $w \in A^*$ be an arbitrary word
that represents the
trivial element $\groupid$ of $G$.  A {\em van Kampen
diagram} $\Delta$ for $w$ with respect to ${\mathcal{P}}$  
is a finite,
planar, contractible combinatorial 2-complex with 
edges directed and
labeled by elements of $A$, satisfying the
properties that the boundary of 
%the infinite region outside of 
$\Delta$ is an edge path labeled by the
word $w$ starting at a basepoint 
vertex $*$ and
reading counterclockwise, and every 2-cell in $\Delta$
has boundary labeled by an element of $R$.
%For any van Kampen diagram
%$\Delta$ with basepoint $*$, let $\pi_\dd:\dd \ra X$
%denote a cellular map such that $\pi_\dd(*)=\ep$ and
%$\pi_\dd$ maps edges to edges preserving both
%label and direction.

Since $w=_G \groupid$, there is a factorization of
$w$ in the free group $F(A)$ as a product of conjugates of
elements of the relator set $R$.  Such a factorization
gives rise to a van Kampen diagram for $w$, with the same number
of 2-cells as factors.  
In general, there may be many different van 
Kampen diagrams for the word $w$.  
%We note that we do not assume that van Kampen diagrams
%in this paper are reduced; that is, we allow adjacent
%2-cells in $\Delta$ to be labeled by the same relator with
%opposite orientations.  

See for example~\cite{bridson} or~\cite{lyndonschupp} 
for an exposition of the theory of van Kampen diagrams.

%%%%%%%%%%%%%%%%%%%%%%%%%%%%%%%%%%%%%%%%%%%%%%%%%%%%%%%%%%%%%%%%%%%%%%%%%%%%%
%%%%%%%%%%%%%%%%%%%%%%%%%%%%%%%%%%%%%%%%%%%%%%%%%%%%%%%%%%%%%%%%%%%%%%%%%%%%%

\subsection{Filled box diagrams for Thompson's group $F$}\label{subsec:box}

%%%%%%%%%%%%%%%%%%%%%%%%%%%%%%%%%%%%%%%%%%%%%%%%%%%%%%%%%%%%%%%%%%%%%%%%%%%%%
%%%%%%%%%%%%%%%%%%%%%%%%%%%%%%%%%%%%%%%%%%%%%%%%%%%%%%%%%%%%%%%%%%%%%%%%%%%%%

We consider the
generating set $A=\{x^{\pm 1},y^{\pm 1}\}$ of $F$,
and construct van Kampen diagrams using
two presentations of $F$ with this generating set, namely
\begin{eqnarray*}
{\mathcal{P}} &:=& \langle A \mid 
[y,xyx^{-2}], [y,x^{2}yx^{-3}] \rangle, \\
{\mathcal{P}}' &:=& \langle A \mid 
\{[y,x^{i}yx^{-i-1}] \mid i \ge 1\} \rangle.
\end{eqnarray*}

As in previous sections, 
let $e$ be a directed edge, with label $y$, 
in the Cayley graph $\Gamma=\Gamma_A(F)$ that does not lie in the tree $T$
determined by the normal form set $\Nfgs$.
Let 
$$
\gamma:=\nf{e_-}=x^{i_n}y^{\epsilon_n}\dots x^{i_1}y^{\epsilon_1}x^{i_0}
$$
and $\gamma':=\nf{e_+}$. 
The vector of cumulative $x$-exponents for $\gamma$ is
$ s(\gamma) = (s_n,\dots,s_0)$, % and $s(\gamma') = (s'_{n'},\dots,s'_0)$,
the cutoff point is $m:=m(\gamma)$, %and $m':= m'(\gamma')$,
and the size sequence is $\sigma(e)=(s_0,\dots,s_{m-1})$.

Consider the word $z=z(e):=\gamma y (\gamma')^{-1}$ in $A^*$,
where $(~)^{-1}$ denotes a formal inverse.
The word $z$ represents the identity element of $F$, and so 
for each of the presentations of $F$ with generating set $A$, there is
a van Kampen diagram for $z$.  We describe a procedure
for building a specific
van Kampen diagram $\Delta_z$ for $z$ with respect to the presentation 
$\mathcal{P}$
%, that we call the {\em box diagram for $Z$}, 
as follows.

First we use the analysis in Section~\ref{subsec:multbyy}
to build a van Kampen diagram $\Delta_z'$ for $z$ over the presentation
$\mathcal{P}'$.  
Each of the $m$ phases of the standard $\Sigma$-rewriting 
from the word $\gamma y$ to the word $\gamma'$
consists of applying a single defining relation from the
presentation $\mathcal{P}'$.  In particular, for 
$0 \le k \le m-1$, the relation applied is
$[y,x^{s_k}yx^{-s_k-1}]$.  
We draw the 2-cell for this relation as a rectangle, with
top and bottom labeled $y^{\epsilon_{k+1}}$ and pointing to the right,
with left and right sides labeled $x^{s_k} y x^{-s_k-1}$ and
oriented from bottom to top. 
The diagram $\Delta_z'$ consists of a 1-dimensional edge path
for a common prefix of $\gamma y$ and $\gamma'$, together with 
these $m$ rectangles, or ``boxes'', in a horizontal sequence
with the right side of the $k$-th 2-cell attached to the 
left side of the $(k-1)$-th 2-cell along the vertical $y$ edges, 
also gluing as many $x$ and $x^{-1}$ edges as possible along the
two vertical sides.  
An illustration of the diagram $\Delta_z'$ in the case
that $\gamma = x^2y^{-1}xy^{-1}x^{-2}yx^4$ is given in
Figure~\ref{f:deltaprime}.
\begin{figure}[!ht]
\includegraphics[width=3in,height=2in]{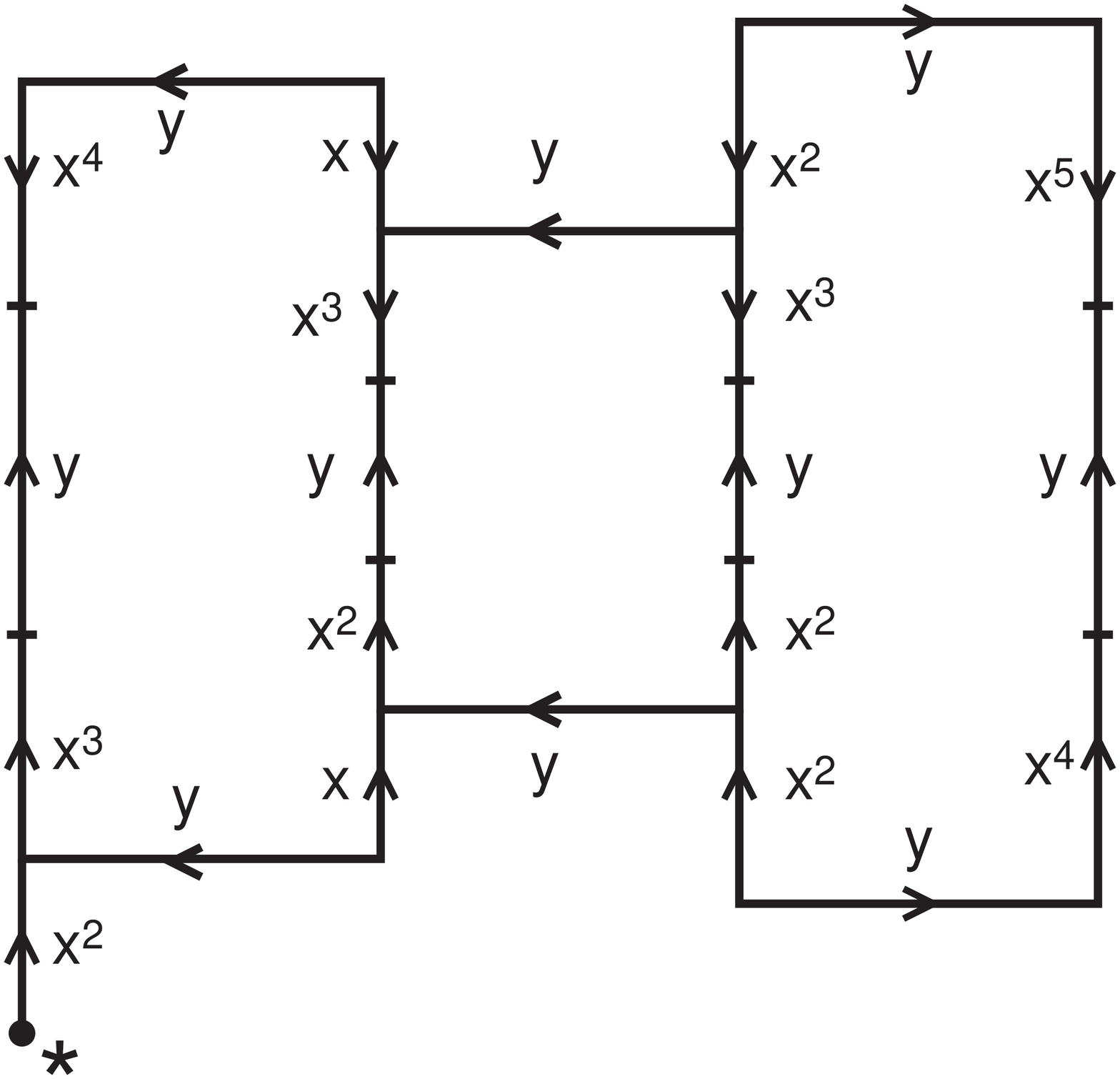}
	\caption{$\Delta_z'$ with $\gamma = x^2y^{-1}xy^{-1}x^{-2}yx^4$}
	\label{f:deltaprime}
\end{figure}

Next we build an iterative procedure for filling a rectangular
2-cell labeled by a relator of the presentation
$\mathcal{P}'$ with a van Kampen diagram
over the presentation $\mathcal{P}$.  Any rectangular box 
labeled by the relator $[y,x^{i}yx^{-i-1}]$ with $i \in \{1,2\}$
of $\mathcal{P}'$ is also a 2-cell labeled by a relator of $\mathcal{P}$,
and so is left unchanged.  For any $i \ge 3$, the rectangular box
labeled by the relator $[y,x^{i}yx^{-i-1}]$ can be filled using
two 2-cells with boundary labels $[y,xyx^{-2}]$, two copies
of the van Kampen diagram over $\mathcal{P}$ for
the word $[y,x^{i-2}yx^{-(i-2)-1}]$, and one copy of the
van Kampen diagram over $\mathcal{P}$ for
the word $[y,x^{i-1}yx^{-(i-1)-1}]$;  see Figure~\ref{f:boxj}
for an illustration of the resulting van Kampen
diagram over $\mathcal{P}$ for
the word $[y,x^{i}yx^{-i-1}]$.
\begin{figure}[!ht]
\includegraphics[width=3in,height=2in]{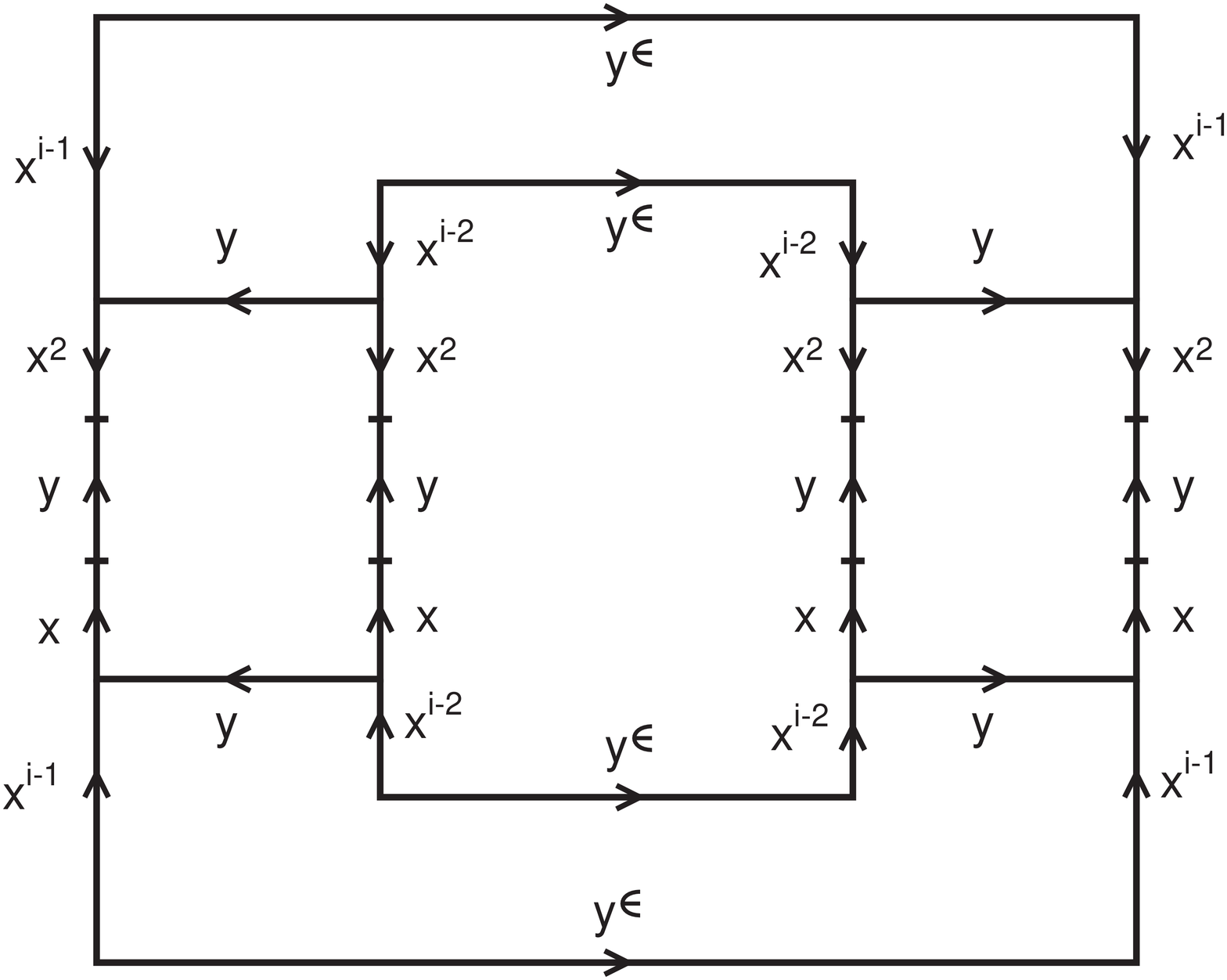}
	\caption{Van Kampen diagram for the relator $[y,x^{i}yx^{-i-1}]$ with $i \ge 3$}
	\label{f:boxj}
\end{figure}

Finally, the van Kampen diagram $\Delta_z$ for the
word $z=\gamma y (\gamma')^{-1}$ is obtained from the diagram
$\Delta_z'$ by replacing each 2-cell with the van Kampen diagram
over the presentation $\mathcal{P}$ built by this iterative procedure.
We refer to the diagram $\Delta_z$ as the {\em filled box diagram for the edge $e$}.

We next count the number of 2-cells in the filled box diagram $\Delta_z$ for $e$.
Note that this box diagram may not be
reduced; that is, there may be adjacent
2-cells in $\Delta_z$ that are labeled by the same relator with
opposite orientations.  Hence we may not have a van Kampen
diagram of minimal possible area (number of 2-cells) 
for the word $z$; we simply have a
diagram that is built in a canonical way.

For all $i \ge 1$, let $C(i)$ denote the number of 2-cells
in the van Kampen diagram over $\mathcal{P}$ for the word $[y,x^{i}yx^{-i-1}]$
built by our iterative procedure.  Then
$C(1)=1$, $C(2) = 1$, and for all $i \ge 3$ we have
$C(i) = C(i-1)+2C(i-2)+2$. That is, the numbers $C(i)$ give
the sequence of positive integers defined in Section~\ref{subsec:weight}. 
Now the number of 2-cells in the filled box diagram $\Delta_z$ associated
to the edge $e$ is the sum of the numbers of 2-cells from presentation
$\mathcal{P}$ that are used to
fill the rectangles in the diagram $\Delta_z'$; that is,
the number of 2-cells in the filled box diagram $\Delta_z$ is
$$
\sum_{k=0}^{m-1} C(s_k),
$$
which is our definition of the weight of $e$
in Section~\ref{subsec:weight}.

%%%%%%%%%%%%%%%%%%%%%%%%%%%%%%%%%%%%%%%%%%%%%%%%%%%%%%%%%%%%%%%%%%%%%%%%%%%%%
%%%%%%%%%%%%%%%%%%%%%%%%%%%%%%%%%%%%%%%%%%%%%%%%%%%%%%%%%%%%%%%%%%%%%%%%%%%%%

\section{Thompson's group $F$ is autostackable}\label{sec:proof}

%%%%%%%%%%%%%%%%%%%%%%%%%%%%%%%%%%%%%%%%%%%%%%%%%%%%%%%%%%%%%%%%%%%%%%%%%%%%%
%%%%%%%%%%%%%%%%%%%%%%%%%%%%%%%%%%%%%%%%%%%%%%%%%%%%%%%%%%%%%%%%%%%%%%%%%%%%%

In this final section we prove that Thompson's group $F$ is autostackable.
As in earlier sections, let 
$\Gamma$ be the Cayley graph of Thompson's group $F$ with respect
to the generating set $A=\{x^{\pm 1},y^{\pm 1}\}$ and let $T$ be the tree 
associated to the set $\Nfgs$ of normal forms from Guba and Sapir's
rewriting system $\Sigma$ (described in Section~\ref{subsec:gsrs}). 

We define a flow function 
$\Phi\colon \vec E\to\vec P$ as follows.
Let $e$ be a directed edge of $\Gamma$.
If $e$ lies on $T$, then $\Phi(e):=e$.
%(as required for of a flow function). 
Otherwise, by Lemma~\ref{lem:yyinv1},
we have $e=e_{\gamma,y^b}$ with 
\[
 \gamma \equalword  py^{\epsilon}x^{i}\in\Nfgs,
\]
where $b,\epsilon\in\{1,-1\}$, and $\gamma y^{b}\notin\Nfgs$.
In this case, we define $\Phi(e)$ to be the directed path starting
at $e_-$ labeled by
\[\lab(\Phi(e)) := 
\begin{cases}
x^{-1}y^{-1}xyx^{-2}yx^2 & \text{ if } b=1 \text{ and } i > 2 \\
x^{-i}y^{-\epsilon}x^{i}yx^{-i-1}y^{\epsilon}x^{i+1} & 
                   \text{ if } b=1 \text{ and } 1 \le i \le 2 \\
x^{-2}y^{-1}x^2y^{-1}x^{-1}yx & \text{ if } b=-1 \text{ and } i > 3 \\
x^{-i}y^{-\epsilon}x^{i}y^{-1}x^{-i+1}y^{\epsilon}x^{i-1} & 
               \text{ if } b=-1 \text{ and } 2 \le i \le 3. \\
\end{cases}
\]

%\begin{enumerate}
%	\item[(1)] If $b=1$ and $i=2$ we define $\Phi(e)$ so that 
%	$$\lab(\Phi(e))\equalword  x^{-2}y^{-\epsilon}x^{2}yx^{-3}y^{\epsilon}x^{3}.$$

%	\item[(2)] If $b=1$ and $i\ne 2$ we define $\Phi(e)$ so that 
%	$$\lab(\Phi(e))\equalword  x^{-1}y^{-1}xyx^{-2}yx^2.$$

%	\item[(3)] If $b=-1$ and $i=3$ we define $\Phi(e)$ so that 
%	$$\lab(\Phi(e))\equalword  x^{-3}y^{-\epsilon}x^{3}y^{-1}x^{-2}y^{\epsilon}x^{2}.$$

%	\item[(4)] If $b=-1$ and $i\ne 3$ we define $\Phi(e)$ so that 
%	$$\lab(\Phi(e))\equalword  x^{-2}y^{-1}x^2y^{-1}x^{-1}yx.$$	
%\end{enumerate}

\noindent Note that the image of every directed edge $e$ of $\Gamma$ is a 
directed path from $e_-$ to $e_+$.  Hence the property of
fixing the tree in the definition of
flow function holds for the map $\Phi$.  

Next we show that the termination property holds for $\Phi$, in order to complete
the proof that $\Phi$ is a flow function.  
To do this, we use the well-founded strict partial order $\prec$ defined in Section~\ref{sec:order}.  
In particular, it suffices to prove the claim:
\begin{itemize}
\item[($*$)]\label{claimf3} {\em For every pair of directed edges 
$e,e'$ in $\Gamma$ such that $e$ is not in $T$ and
$e'$ is in the path $\Phi(e)$, either $e'$ lies in the tree $T$, or
$e' \prec e$.}
\end{itemize}
%, then there cannot be a sequence of edges outside of $T$

To begin the proof of this claim,
we note that for each edge $e$ we have $\Phi(e^{-1}) = \Phi(e)^{-1}$, and 
%This, together with the fact that 
every edge has the same weight as its inverse.
%This enables us to reduce the further discussion solely to 
Hence it suffices to prove this claim for any edge $e$ (not in $T$)
of the form $e=e_{\gamma,y}$. 

Let $g$ be the element of $F$ with $g=_F e_-$, and write
the normal form $\gamma :=\nf{g}$ 
(using Lemma~\ref{lem:yyinv1}) in the form 
\[ 
 \gamma = x^{i_n}y^{\epsilon_n} \dots x^{i_1}y^{\epsilon_1}x^{i_0} = 
      py^{\epsilon_1}x^{i_0},
\]
where each $\epsilon_j \in \{1,-1\}$, $n \geq 1$, and $i_0 \geq 1$. 
Let $g' =_F e_+$, and (again using Lemma~\ref{lem:yyinv1}) 
write the normal form $\gamma':=\nf{g'}$ in the form
\[ 
 \gamma' = x^{j_{n'}}y^{\varepsilon_{n'}} \dots x^{j_1}y^{\varepsilon_1}x^{j_0} = 
      p'y^{\varepsilon_1}x^{j_0},
\]
where each $\varepsilon_j \in \{1,-1\}$, $n' \geq 1$, and $j_0 \geq 2$.

Note also that Lemma~\ref{lem:yyinv1} says that all edges in 
the path $\Phi(e)$ labeled $x$ or $x^{-1}$ are in the tree $T$.
We consider the three other edges in $\Phi(e)$ labeled $y$ or $y^{-1}$,
in two cases.

%----------------------------------------

%%%%%%%%%%%%%%%%%%%%%%%%%%%%%%%%%%%%%%%%%%%%%%%%%%%%%%%%%%%%%%%%%%%%%%%%%%%%%
%%%%%%%%%%%%%%%%%%%%%%%%%%%%%%%%%%%%%%%%%%%%%%%%%%%%%%%%%%%%%%%%%%%%%%%%%%%%%

\medskip

\noindent{\bf Case I.} {\em Suppose that $1 \leq i_0 \leq 2$.}

%%%%%%%%%%%%%%%%%%%%%%%%%%%%%%%%%%%%%%%%%%%%%%%%%%%%%%%%%%%%%%%%%%%%%%%%%%%%%
%%%%%%%%%%%%%%%%%%%%%%%%%%%%%%%%%%%%%%%%%%%%%%%%%%%%%%%%%%%%%%%%%%%%%%%%%%%%%

The first $y^{\pm 1}$ edge on $\Phi(e)$,
namely $e_{py^{\epsilon_1},y^{-\epsilon_1}}$, 
is the inverse of an edge on the normal form path in $\Gamma$
starting at $\groupid$ and labeled by $\gamma$,
and hence is in the normal form tree $T$.  

A similar situation occurs for the third $y^{\pm 1}$ edge on $\Phi(e)$:
Since $m(e) \ge 1$, Lemma~\ref{lem:multbyy} shows that the normal
form $\gamma'=p'y^{\varepsilon_1}x^{j_0}$ labeling the path 
in $\Gamma$ from $\groupid$ to $e_+$ satisfies
$j_0=i_0+1$ and $\varepsilon_1=\epsilon_1$.  
Thus the third $y^{\pm 1}$ edge on $\Phi(e)$
is $e_{p',y^{\varepsilon_1}}$, which lies
on the path labeled $\gamma'$ in $T$.

Finally we consider the second $y^{\pm 1}$ edge $e'$ on $\Phi(e)$.
This edge has label $y$
and its initial vertex 
$e'_- =_F px^{i_0}$ has normal form
$x^{i_n}y^{\epsilon_n} \dots y^{\epsilon_2}x^{i_1+i_0}$.
If this edge is not in the tree, then Lemma~\ref{lem:yyinv1} says
that $n \ge 2$ and $i_1+i_0 \ge 1$.  Now the size sequence of $e'$ is
$\sigma(e')=(s_{1},\dots,s_{m-1})$, 
%$(s_{m-1},\dots,s_1)$,
and so $W(e')<W(e)$; hence $e' \prec e$.

%In the case that $i_0=1$ and $\epsilon_1=-1$,
%the initial vertex of $e'$ is 
%$e'_- =_F py^{-2}x$, which has normal form
%$x^{i_n}y^{\epsilon_n} \dots x^{i_1}y^{-2}x$.
%Then $e'$ is not in the tree $T$, and
%the size vector of $e'$ is $(s_{m-1},\dots,s_1,s_1)$.
%Hence $e \prec e'$ and we have a problem!

%This is a straightforward case, since the first and the third 
%$y^\pm$-edge on $\Phi(e)$  are always on the tree, while the second one, 
%if it is not on the tree, has size vector $(s_{m-1},\dots,s_1)$. 

%----------------------------------------

\medskip

\noindent{\bf Case II.} {\em Suppose that $i_0 \geq 3$.} 
%For the duration of this case we assume that $i_0 \geq 3$. 

The function $\Phi$ replaces the edge $e$ by the path from $g$ to $g'$ of length 9 
labeled by $x^{-1}y^{-1}xyx^{-2}yx^2$. For easier discussion, consider 
the diagram in Figure~\ref{f:i3}. In that diagram, all dotted, double-tip edges, 
that is, the edges $e$, $e_1$, $e_2$, and $e_3$, are labeled by $y$ and all 
other edges by $x$. 
\begin{figure}[!ht]
	\[
	\xymatrix{
		&& 
		\bullet \ar@{->}[r] \ar@{..>>}^{e_1}[d] \ar@{}|<<{f}[l] & 
		\bullet \ar@{..>>}^{e_2}[r] \ar@{}|<<{h}[d] & 
		\bullet \ar@{<-}[r]	& \bullet \ar@{<-}[r] & \bullet \ar@{..>>}^{e_3}[d] && \\
		\bullet \ar@{->}[r] & \bullet \ar@{->}[r] & \bullet \ar@{->}[r] \ar@{}|<<{gx^{-1}}[d] & 	
		\bullet \ar@{..>>}_{e}[r] \ar@{}|<<{g}[d]& 
		\bullet \ar@{<-}[r] \ar@{}|<<{g'}[d] & \bullet \ar@{<-}[r] & \bullet \ar@{<-}[r] \ar@{}|<<{g'x^{-2}}[d] & \bullet \ar@{<-}[r] & \bullet
		\\
		&&&&&&&&
	}
	\]
	\caption{$e$ and $\Phi(e)$ when $i_0 \geq 3$}
	\label{f:i3}
\end{figure}
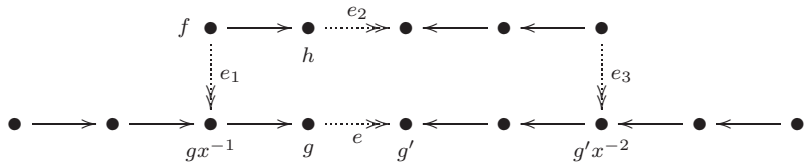
The path $\Phi(e)$ connects $g$ to $g'$ by using the shortest path in the diagram, 
consisting of 9 edges, that does not use the edge $e$ 
(thus, it uses the edges $e_1^{-1}$, $e_2$, and  $e_3$, along with 
6 $x^{\pm 1}$-edges which (as already noted above) are in the normal form tree $T$). 

\begin{lem}\label{lem:f3i3}
None of the edges $e_1^{-1}$, $e_2$, and  $e_3$ is in the normal form tree, and each has weight strictly smaller than the weight of $e$. 
\end{lem}

\begin{proof}
Let $s(\gamma)=(s_n,\dots,s_0)$ be the sequence of partial sums of the $x$-exponents in $\gamma$, and let $m:=m(e)= \min\{k \mid s_k \leq 0  \}$. Denote $f:=gx^{-1}y^{-1}$ and $h:=fx =_F gx^{-1}y^{-1}x$. 
Note that
\[
 \nf{(gx^{-1})} = x^{i_n}y^{\epsilon_n} \dots x^{i_1}y^{\epsilon_1}x^{i_0-1}.
\]

Since $i_0 - 1 \geq 2$, the edge $e_1^{-1}$ is not on the normal form tree. But this means that the normal form of $f$ ends with at least one $x$ (in fact, exactly $i_0-2$ according
to Lemma~\ref{lem:multbyyy}) and the normal form of $h$ ends with at least two (in fact $i_0-1$), which implies that $e_2$ is not on the normal from tree. The normal form of $g'$ ends with at least 4 letters $x$ (Lemma~\ref{lem:multbyy} says $g'$ ends with exactly $i_0+1$), which implies that $g'x^{-2}$ ends with at least 2 (exactly $i_0-1$), and so $e_3$ is not on the normal form tree either. 

We have 
\[
m_1 := m(e_1^{-1}) = m'(\nf{gx^{-1}}) 
= \min\{k \mid i_k+\dots+(i_0-1) \leq 1 \} = \min\{k \mid s_k \leq 2 \}
\]
when the set is nonempty and, otherwise, $m_1=n$. Hence $m_1 \leq m$. 
Now the vector $s(\nf{gx^{-1}})$ of cumulative $x$-exponents
is $s(\nf{gx^{-1}})=(s_{n}-1,...,s_{0}-1)$, and so 
Lemmas~\ref{lem:multbyyy} and~\ref{l:same} say that the sequence 
of rule sizes associated to the edge $e_1$ is
$\sigma(e_1)=\sigma(e_1^{-1})=(s_0-2,...,s_{m_1-1}-2)$.
Now $s_0=i_0 \ge 3$ (since we are in Case~II),
and so $m_1 \ge 1$; moreover,
for each $k=0,...,m_1-1$ we have $s_k \ge 3$.
Using the fact that $C$ is a nondecreasing
sequence of positive integers (for the inequalities) 
and the recurrence relation defining $C$ (in the last equality)  yields
\begin{eqnarray*}
W(e) -W(e_1) &=& \sum_{k=0}^{m-1} C(s_k) - \sum_{\ell=0}^{m_1-1} C(s_\ell-2)
= \sum_{k=0}^{m_1-1} [C(s_k) - C(s_k-2)] + \sum_{k=m_1}^{m-1} C(s_k)\\
&\ge&    C(s_0) - C(s_0-2) = C(s_0-1)+C(s_0-2) + 2 > 0.
\end{eqnarray*}
Therefore $W(e) > W(e_1)$.

By using Table~\ref{t:exponentsyy} we can calculate the cumulative $x$-exponents vector $s(f)$ directly from the cumulative $x$-exponents vector $s(gx^{-1})$, and from there we compute 
the cumulative $x$-exponents vector $s(h)$ 
simply by adding 1 in each coordinate of $s(f)$. Partial results are given in Table~\ref{t:phi} (with the no $y^{-1}$-cancellation case on the left 
and the $y^{-1}$-cancellation case on the right); 
the cancellation cases in that table refer to crossing the edge $e_1^{-1}$, that is multiplying $gx^{-1}$ by $y^{-1}$. 
\begin{table}[!ht]
\begin{tabular}{cc}
	$
	\begin{array}{|c|c|c|c|}
	\hline 
	k & s(f)_k & s(h)_k & s(g)_k \\
	\hline 
	0  & s_0-2    & s_0-1 & s_0 \\
	1  & s_1-2    & s_1-1 & s_1 \\
	\vdots & \vdots & \vdots & \vdots \\
	m_1-1 & s_{m_1-1}-2 & s_{m_1-1}-1 & s_{m_1-1}\\
	m_1   & \mathbf{0} & 1 & s_{m_1} \\
	m_1+1 & s_{m_1}-1  & s_{m_1} & s_{m_1+1} \\
	m_1+2 & s_{m_1+1}-1 & s_{m_1+1} & s_{m_1+2} \\
	\vdots & \vdots & \vdots & \vdots \\
	\vdots & \vdots & \vdots & \vdots \\
	m   & s_{m-1}-1 & s_{m-1} & \mathbf{s_{m}}\\
	m+1 & s_{m}-1 & \mathbf{s_{m}} & \\	
	\hline
	\end{array}
	$
	&
	$
	\begin{array}{|c|c|c|c|}
	\hline 
	k & s(f)_k & s(h)_k & s(g)_k\\
	\hline 
	0  & s_0-2    & s_0-1 & s_0\\
	1  & s_1-2      & s_1-1 & s_1 \\
	\vdots & \vdots & \vdots & \vdots \\
	m_1-1 & s_{m_1-1}-2 & s_{m_1-1}-1 & s_{m_1-1} \\
	m_1   & s_{m_1+1} -1 & s_{m_1+1} & s_{m_1} \\
	m_1+1 & s_{m_1+2}-1 & s_{m_1+2} & s_{m_1+1}\\
	\vdots & \vdots & \vdots & \vdots \\
	m-2 & s_{m-1}-1 & s_{m-1} & s_{m-2}\\	
	m-1 & s_{m}-1 & \mathbf{s_{m}} & s_{m-1}\\
	m   &         &       & \mathbf{s_{m}}\\	
	\hline
	\end{array}
	$
\end{tabular}
\caption{Cumulative $x$-exponents of $f$, $h$, and $g$: the no $y^{-1}$-cancellation case  (left) and the $y^{-1}$-cancellation case (right)}
\label{t:phi}
\end{table}

We claim that, in each $s(h)_k$ and $s(g)_k$ column in Table~\ref{t:phi}, the boldface entry is the first nonpositive entry.
%(In the $s(f)_k$ column in the $y^{-1}$ cancellation case, no entry is boldfaced.) 
For $g$ this is immediate from the definition of the cutoff point $m(g)=m$.
%and for $f$, this follows from the fact that Lemma~\ref{l:same} says that
%$m(f)=m(e^{-1})=m_1$.
%Thus, we just need to prove this claim for $h$. 
For $k=0,\dots,m_1-1$, by the definition of $m_1$, we have $s_k \geq 3$, which implies that $s(h)_k=s_k-1 \geq 2$. By the definition of $m$, we also have that $s_k \geq 1$ for $k \leq m-1$ and this completes our claim. Note that we just proved that $m(h)=m+1$ in the no $y^{-1}$ cancellation case, and $m(h)=m-1$ in the $y^{-1}$-cancellation case. 

Since the cumulative $x$-exponents of $g$ and $h$ before the index $m=m(g)$ 
and $m(h)$, respectively, represent the rule sizes associated to the 
edges $e$ and $e_2$, 
we again use them to compare weights.  In particular, the difference
in weights $W(e)-W(e_2)$ is the difference between the sum of the 
entries in the sequence $C$ (defined in Section~\ref{subsec:weight})
with indices in the $s(g)_k$ column before $m(g)$
and the sum of the $C$ sequence
entries with indices in the $s(h)_k$ column before $m(h)$;  
hence entries appearing in both columns will cancel.
In more detail, in
the no $y^{-1}$ cancellation case we have
\begin{eqnarray*}
 W(e) - W(e_2) &=& \sum_{k=0}^{m-1} C(s_k) - \sum_{\ell=0}^{m(h)-1} C(s(h)_\ell)\\
&=& \sum_{k=0}^{m_1-1} [C(s_k)-C(s_k-1)] + C(s_{m_1}) -C(1) +
\sum_{k=m_1+1}^{m-1} [C(s_k) - C(s_{k-1})] - C(s_{m-1})\\
&=& \sum_{k=0}^{m_1-1} [C(s_k)-C(s_k-1)] -1 \\
&\ge& C(s_0)-C(s_0-1)-1 =  2C(s_0-2) +2 - 1 > 0
\end{eqnarray*}
where in the inequalities we used that $C$ is a nondecreasing sequence
of positive integers, and between the two inequalities we apply
the recurrence relation defining the sequence $C$. 
A similar telescoping sum occurs in the $y^{-1}$ cancellation case, yielding
\[
W(e) - W(e_2) = \sum_{k=0}^{m_1-1} [C(s_k)-C(s_k-1)] + C(s_{m_1})
\ge C(s_0)-C(s_0-1) >0.
\]
Therefore $W(e)>W(e_2)$.

%The claim that $W(e)>W(e_1)$ is clear, since $\sigma(e)$ is not shorter than $\sigma(e_2)$ %and it has larger corresponding entries. 

Finally, since $i_0 \ge 3$, Lemma~\ref{lem:multbyy} says that
the normal form $\nf{g'}$ has a suffix $x^{j_0}$ with $j_0 \ge 4$,
and so the normal form $\nf{g'x^{-2}}$ is obtained %from $\nf{g'}$ 
by removing an $x^2$ suffix from $\nf{g'}$. 
Hence $s(g'x^{-2})_k=s(g')_k-2$ for all $k$, and so $m(e_3) \le m(e)$.
Then the size sequence $\sigma(e)$ contains at least as many entries as $\sigma(e_3)$,
and has larger corresponding entries.
%Moreover, the first entry of $s(g')$ satisfies $s(g')_0 \ge 4$.
%reducing the trailing $x$-exponent by 2.  Then the number $m(e)$ of entries
%of $\sigma(e)$ is not less than the number $m(e_3)$ 
%of entries $\sigma(e_3)$, and $\sigma(e)$ has larger corresponding entries. 
Therefore $W(e)>W(e_3)$. 

This completes the proof of Case~II and Lemma~\ref{lem:f3i3}.
\end{proof}

%%%%%%%%%%%%%%%%%%%%%%%%%%%%%%%%%%%%%%%%%%%%%%%%%%%%%%%%%%%%%%%%%%%%%%%%%%%%%
%%%%%%%%%%%%%%%%%%%%%%%%%%%%%%%%%%%%%%%%%%%%%%%%%%%%%%%%%%%%%%%%%%%%%%%%%%%%%

\medskip

This also completes the proof of the claim ($*$) (on page~\pageref{claimf3}).
Therefore the function $\Phi$ satisfies the termination property and is a flow function
for Thompson's group $F$.  This flow function is bounded, with bounding constant
$K=13$.

It remains to show that the set $\Graph{\Phi}$ is a regular language.
For the flow function $\Phi$, this graph is
\begin{eqnarray*}
\Graph{\Phi} &=&  
\cup_{a \in A} \left\{(u,  a,  a) \mid ua \in \Nfgs \text{ or } u \in (\Nfgs \cap A^*a^{-1})\right\}  \\
& & \cup \left(\cup_{\epsilon\in\{1,-1\}} \cup_{i \in \{1,2\}}
\left\{(uy^\epsilon x^i,  y,  x^{-i}y^{-\epsilon}x^{i}yx^{-i-1}y^{\epsilon}x^{i+1})
   \mid uy^{\epsilon}x^i \in \Nfgs\right\} \right) \\
& & \cup \left(\cup_{\epsilon\in\{1,-1\}} \cup_{i \in \{2,3\}}
\left\{(uy^\epsilon x^i,  y^{-1},    
    x^{-i}y^{-\epsilon}x^{i}y^{-1}x^{-i+1}y^{\epsilon}x^{i-1}) \mid
  uy^\epsilon x^i \in \Nfgs\right\} \right) \\
& & \cup \left\{(ux,  y,   x^{-1}y^{-1}xyx^{-2}yx^2) \mid
u \in \Nfgs \cap (A^*y^{\pm 1}A^*x^2)\right\}  \\
& & \cup \left\{(ux^2, y^{-1},  x^{-2}y^{-1}x^2y^{-1}x^{-1}yx) \mid
u \in \Nfgs \cap (A^*y^{\pm 1}A^*x^2)\right\}\\
&=&  
\cup_{a \in A} 
\left(\Nfgs_{a} \cup (\Nfgs \cap A^*a^{-1})\right) \times \{a\} \times \{a\} \\
& & \cup \left(\cup_{\epsilon\in\{1,-1\}} \cup_{i \in \{1,2\}}
(\Nfgs \cap {A^*y^{\epsilon}x^i}) \times \{y\} \times 
     \{x^{-i}y^{-\epsilon}x^{i}yx^{-i-1}y^{\epsilon}x^{i+1}\}\right)\\
& & \cup \left(\cup_{\epsilon\in\{1,-1\}} \cup_{i \in \{2,3\}}
(\Nfgs \cap {A^*y^\epsilon x^i}) \times \{y^{-1}\} \times    
    \{x^{-i}y^{-\epsilon}x^{i}y^{-1}x^{-i+1}y^{\epsilon}x^{i-1}\}\right) \\
& & \cup \left(\Nfgs \cap (A^*y^{\pm 1}A^*x^3)\right)  
  \times \{y\} \times  \{x^{-1}y^{-1}xyx^{-2}yx^2\}  \\
& & \cup \left(\Nfgs \cap (A^*y^{\pm 1}A^*x^4)\right) \times
  \{ y^{-1}\} \times   \{x^{-2}y^{-1}x^2y^{-1}x^{-1}yx \}.
\end{eqnarray*}

Since the language $\Nfgs$ is regular, Lemma~\ref{lem:regclosure}
says that the quotient language $\Nfgs_{a}$ 
is also regular.  Now closure of the class
of regular languages under union, intersection, concatenation, etc., shows that
the left factor in each triple is a regular language.  Since any
set consisting of a single word is also regular,  Lemma~\ref{lem:regclosure}
shows that each of the above products of three regular languages is also regular.
Finally, closure under union again shows that $\Graph{\Phi}$ is regular.

Therefore $\Phi$ is a regular bounded flow function, and so Thompson's
group $F$ is autostackable.  The prefix-rewriting system in the statement
of Theorem~\ref{thm:fisautostk} is the system associated to the
flow function $\Phi$, and hence is convergent.  This completes
the proof of Theorem~\ref{thm:fisautostk}.

%%%%%%%%%%%%%%%%%%%%%%%%%%%%%%%%%%%%%%%%%%%%%%%%%%%%%%%%%%%%%%%%%%
%%%%%%%%%%%%%%%%%%%%%%%%%%%%%%%%%%%%%%%%%%%%%%%%%%%%%%%%%%%%%%%%%%

\section*{Acknowledgments}

%%%%%%%%%%%%%%%%%%%%%%%%%%%%%%%%%%%%%%%%%%%%%%%%%%%%%%%%%%%%%%%%%%
%%%%%%%%%%%%%%%%%%%%%%%%%%%%%%%%%%%%%%%%%%%%%%%%%%%%%%%%%%%%%%%%%%

The third author was partially supported by grants from
the National Science Foundation (DMS-1313559) and
the Simons Foundation (Collaboration Grant number 581433).

\end{document}